\documentclass[12pt,a4paper,twoside]{article}
\usepackage{a4wide,amsmath,amsbsy,amsfonts,amssymb,stmaryrd,amsthm,mathrsfs,graphicx}
\newcommand{\cev}[1]{\reflectbox{\ensuremath{\vec{\reflectbox{\ensuremath{#1}}}}}}           

\newcommand\Z{{\mathbb Z}}

\newcommand\R{{\mathbb R}}
\newcommand\C{{\mathbb C}}

\newcommand\cP{{\mathscr P}}
\newcommand\ra{\rightarrow}

\newcommand\ilim{\lim\limits_{\longleftarrow}\,}

\newcommand\Sp{\mathrm{Sp}}

\newcommand\aut{\mathrm{Aut}}
\newcommand\out{\mathrm{Out}}
\newcommand\inn{\mathrm{inn}}

\newcommand\lra{\longrightarrow}
\newcommand\hookra{\hookrightarrow}
\newcommand\tura{\twoheadrightarrow}
\newcommand\da{\downarrow}

\newcommand\dd{\partial}
\renewcommand{\hom}{\mathrm{Hom}}

\renewcommand{\Im}{\mathrm{Im}} 
\newcommand\sr{\stackrel}
\newcommand\st{\scriptstyle}
\newcommand\sst{\scriptscriptstyle}
\newcommand\cGG{\check{\GG}}
\newcommand\hGG{\hat{\GG}}

\newcommand\hp{\hat{\pi}}

\newcommand\hC{\hat{C}}

\newcommand\ssm{\smallsetminus}
\newcommand\ol{\overline}

\newcommand\ccM{\overline{\mathcal M}}
\newcommand\cM{{\mathcal M}}
\newcommand\cN{{\mathcal N}}

\newcommand\cB{{\mathcal B}}
\newcommand\cC{{\mathscr C}}

\newcommand\cE{{\mathscr E}}
\newcommand\cG{{\mathscr G}}
\newcommand\cS{{\mathscr S}}

\newcommand\cT{{\mathscr T}}
\newcommand\GG{\Gamma}
\newcommand\ld{\lambda}
\newcommand\Ld{\Lambda}
\newcommand\wh{\widehat}
\newcommand\wt{\widetilde}
\newcommand\wT{\wh{T}}
\newcommand\wM{\wh{\cM}}

\newcommand\hd{\hat{\Delta}}
\newcommand\od{\dot{\delta}}

\newcommand\td{\tilde}
\newcommand\sg{\sigma}
\newcommand\Sg{\Sigma}
\newcommand\gm{\gamma}
\newcommand\bt{\bullet}

\def\co{\colon\thinspace}

\newtheorem{theorem}{Theorem}[section]
\newtheorem{corollary}[theorem]{Corollary}
\newtheorem{proposition}[theorem]{Proposition}
\newtheorem{lemma}[theorem]{Lemma}

\newtheorem{question}[theorem]{Question}

\theoremstyle{definition}
\newtheorem{definition}[theorem]{Definition}   
\newtheorem{remark}[theorem]{Remark}      
\newtheorem{remarks}[theorem]{Remarks}

\begin{document}

\title{Galois coverings of moduli spaces of curves\\ and loci of curves with symmetry}
\author{Marco Boggi}\maketitle

\begin{abstract}
Let $\ccM_{g,[n]}$, for $2g-2+n>0$, be the stack of genus $g$, stable algebraic curves over $\C$,
endowed with $n$ unordered marked points. 

In \cite{L1}, Looijenga introduced the notion of Prym level 
structures in order to construct smooth projective Galois coverings of the stack $\ccM_{g}$.

In \S 2 of this paper, we introduce the notion of {\it Looijenga level structure} which generalizes 
Looijenga construction and provides a tower of Galois coverings of $\ccM_{g,[n]}$ equivalent 
to the tower of all geometric level structures over $\ccM_{g,[n]}$. 

In \S 3, Looijenga level structures are interpreted
geometrically in terms of moduli of curves with symmetry. A byproduct of this characterization is
a simple criterion for their smoothness. As a consequence of this criterion, it is shown that  
Looijenga level structures are smooth under mild hypotheses.

The second part of the paper, from \S 4, deals with the problem of describing the D--M boundary 
of Looijenga level structures. In \S 6, a description is given of the nerve of the D--M boundary of 
abelian level structures. In \S 7, it is shown how this construction can be used to 
"approximate" the nerve of Looijenga level structures. These results are then applied to elaborate  
a new approach to the congruence subgroup problem for the Teichm\"uller 
modular group along the lines of \cite{PFT}.
\newline

\end{abstract}

\section{Level structures over moduli of curves}\label{levels}
In this first section we give a preliminary exposition of the theory of level structures over moduli of 
curves, mostly needed in order to fix notation. So, let $\ccM_{g,n}$, for $2g-2+n>0$, be the stack of 
$n$--pointed, genus $g$, stable algebraic curves over $\C$.  It is a regular connected proper D--M 
stack over $\C$ of dimension $3g-3+n$, and it contains, as an open substack, the stack $\cM_{g,n}$ 
of $n$--pointed, genus $g$, smooth algebraic curves over $\C$. We will keep the same notations to 
denote the respective underlying analytic and topological stacks. 

There is a natural action of the symmetric group $\Sg_n$ on the $n$ labels which mark the curves
parametrized by the stack $\ccM_{g,n}$. This induces an action of $\Sg_n$ on $\ccM_{g,n}$
which is stack-theoretically free. The quotient by this action is denoted by $\ccM_{g,[n]}$ and is the 
stack of genus $g$, stable algebraic curves over $\C$, endowed with $n$ unordered marked points. 
The open substack parametrizing smooth curves is then denoted by $\cM_{g,[n]}$.

Let $S_{g,n}$ be an $n$-punctured, genus $g$ Riemann surface. Then, the universal cover of 
$\cM_{g,n}$, the Teichm{\"u}ller space $T_{g,n}$, is the stack of $n$--pointed, genus $g$, smooth 
complex analytic curves $(\cE\ra{\mathscr U}, s_1,\ldots, s_n)$ endowed with a topological trivialization 
$$\Phi\co S_{g,n}\times{\mathscr U}\sr{\sim}{\ra}\cE\ssm\bigcup_{i=1}^n s_i({\mathscr U})$$ 
over ${\mathscr U}$, where two such trivializations are considered equivalent when they are 
homotopic over ${\mathscr U}$. We then denote by $(\cE\ra{\mathscr U}, s_1,\ldots, s_n,\Phi)$ the 
corresponding object of $T_{g,n}$ or, when ${\mathscr U}$ is just one point, simply by $(E,\Phi)$.

From the existence of Kuranishi families, it follows that the complex analytic stack $T_{g,n}$ is
representable by a complex manifold (cf. \cite{A-C}, for more details on this approach). Then, it is 
not hard to prove that the complex manifold $T_{g,n}$ is contractible and that the natural map of 
complex analytic stacks $T_{g,n}\ra\cM_{g,[n]}$ is a universal cover. Its deck transformation 
group is described as follows. 

Let $\hom^+(S_{g,n})$ be the group of orientation preserving self-homeomorphisms
of $S_{g,n}$ and by $\hom^0(S_{g,n})$ the subgroup consisting of homeomorphisms 
homotopic to the identity. The mapping class group $\GG_{g,[n]}$ is defined to be the 
group of homotopy classes of homeomorphisms of $S_{g,n}$ which preserve the orientation:
$$\GG_{g,[n]}:=\left.\hom^+(S_{g,n})\right/\hom^0(S_{g,n}),$$  
where $\hom_0(S_{g,n})$ is the connected component of the identity 
in the topological group of homeomorphisms $\hom^+(S_{g,n})$. 
This group then is the deck transformation group of the covering $T_{g,n}\ra\cM_{g,[n]}$.
There is a short exact sequence:
$$1\ra\GG_{g,n}\ra\GG_{g,[n]}\ra\Sigma_n\ra 1,$$
where $\Sigma_n$ denotes the symmetric group on the set of punctures of $S_{g,n}$
and $\GG_{g,n}$ is the deck transformation group of the covering $T_{g,n}\ra\cM_{g,n}$. 

There is a natural way to define homotopy groups for topological D--M stacks, as done, for instance, 
by Noohi in  \cite{N1} and \cite{N2}. Therefore, the choice of a point $a=[C]\in\cM_{g,[n]}$ and a 
homeomorphism $\phi\co S_{g,n}\ra C$ identifies the topological fundamental group 
$\pi_1(\cM_{g,[n]},a)$ with the Teichm{\"u}ller modular group $\GG_{g,[n]}$.

The following notation will turn out to be useful later. For a given oriented Riemann surface
$S$ of negative Euler characteristic, let us denote by $\GG(S)$ the mapping class group of $S$ and 
by $\cM(S)$ and $\ccM(S)$, respectively, the D--M stack of smooth complex curves homeomorphic 
to $S$ and its D--M compactification. In particular, $\GG_{g,[n]}:=\GG(S_{g,n})$,
$\cM_{g,[n]}:=\cM(S_{g,n})$ and $\ccM_{g,[n]}:=\ccM(S_{g,n})$. 

Let $\cM_{g,[n]+1}$ be the moduli stack of genus $g$, stable algebraic curves over $\C$, 
endowed with $n$ unordered labels and one distinguished label $P_{n+1}$.
The morphism $\cM_{g,[n]+1}{\ra}\cM_{g,n}$, induced by forgetting $P_{n+1}$, is
naturally isomorphic to the universal $n$-punctured, genus $g$ curve $p\co\cC\ra\cM_{g,[n]}$. 
Since $p$ is a Serre fibration and $\pi_2(\cM_{g,[n]})=\pi_2(T_{g,n})=0$, there is a short exact 
sequence on fundamental groups
$$1\ra\pi_1(\cC_a,\tilde{a})\ra\pi_1(\cC,\tilde{a})\ra\pi_1(\cM_{g,[n]},a)\ra 1,$$
where $\tilde{a}$ is a point in the fiber $\cC_a$.
By a standard argument this defines a monodromy representation:
$$\rho_{g,[n]}\co\pi_1(\cM_{g,[n]},a)\ra \mbox{Out}(\pi_1(\cC_a,\tilde{a})),$$
called the {\it universal monodromy representation}. 

Let us then fix a homeomorphism $\phi\co S_{g,n}\ra\cC_a$ and let
$\Pi_{g,n}$ be the fundamental group of $S_{g,n}$ based in
$\phi^{-1}(\tilde{a})$. Then, the representation $\rho_{g,[n]}$ is identified with
the faithful representation $\GG_{g,[n]}\hookra\out(\Pi_{g,n})$, induced by the 
action, modulo homotopy, of $\GG_{g,[n]}$ on the Riemann surface $S_{g,n}$.

Let us give $\Pi_{g,n}$ the standard presentation:
$$ \Pi_{g,n}=\langle\alpha_1, \dots \alpha_g,\beta_1,\dots,\beta_g, u_1,\dots,u_n|\;
\prod_{i=1}^g[\alpha_i,\beta_i] \cdot u_n\cdots u_1\rangle,$$ 
where $u_i$, for $i=1,\dots,n$, is a simple loop around the  puncture $P_i$. For $n\ge 1$, let
$A(g,n)$ be the group of automorphisms of $\Pi_{g,n}$ which  fix the set of conjugacy classes of
all $u_i$. For $n=0$,  let instead $A(g,0)$ be the image of
$A(g,1)$ in the automorphism group of $\Pi_g:=\Pi_{g,0}$. Finally, let $I(g,n)$ be the group of 
inner  automorphisms of $\Pi_{g,n}$.  With these notations, the Nielsen realization Theorem says 
that the representation $\rho_{g,[n]}$ induces an isomorphism $\GG_{g,[n]}\cong A(g,n)/I(g,n)$.

In this paper, a level structure $\cM^\ld$ is a finite, connected, Galois, {\'e}tale covering of the stack 
$\cM_{g,[n]}$ (by {\'e}tale covering, we mean here an {\'e}tale, surjective, representable morphism of
algebraic stacks), therefore it is also a regular D--M stack.  The {\it level} 
associated to $\cM^\ld$ is the finite index normal subgroup
$\GG^\ld\cong\pi_1(\cM^\ld, a')$ of the Teichm{\"u}ller group $\GG_{g,[n]}$.
 
A level structure $\cM^{\ld'}$ {\it dominates} $\cM^\ld$, if there is a natural {\'e}tale
morphism $\cM^{\ld'}\ra\cM^{\ld}$ or, equivalently, $\GG^{\ld'}\leq\GG^\ld$. To stress the fact that
$\cM^\ld$ is a level structure over $\cM_{g,[n]}$ (respectively, $\cM(S)$), we will sometimes denote it 
by $\cM_{g,[n]}^\ld$ (respectively, $\cM(S)^\ld$).

For a given level $\GG^\ld$ of $\GG_{g,[n]}$, the intersection $\GG^\ld\cap\GG_{g,n}$ is also
a level of $\GG_{g,[n]}$ which we denote by $\GG^\ld_{g,n}$. The corresponding level structure is
denoted by $\cM_{g,n}^\ld$. Equivalently, the level structure $\cM_{g,n}^\ld$ is the pull-back over 
$\cM_{g,n}\ra\cM_{g,[n]}$ of the level structure $\cM^\ld$.

The most natural way to define levels is provided by the universal monodromy representation 
$\rho_{g,[n]}$. In general, for a subgroup $K\le\Pi_{g,n}$, which is invariant under $A(g,n)$ 
(in such case, we simply say that $K$ is {\it invariant}), let us define the representation:
$$\rho_K\co\GG_{g,[n]}\ra \mbox{Out}(\Pi_{g,n}/K),$$
whose kernel we denote by $\GG^K$. When $K$ has finite index in $\Pi_{g,n}$, then
$\GG^K$ has finite index in $\GG_{g,[n]}$ and is called the {\it geometric level} associated to
$K$. The corresponding level structure is denoted by $\cM^{K}$. 

Of particular interest are the levels defined by the kernels of the representations:
$$\rho_{(m)}\co\GG_{g,[n]}\ra \mbox{Sp}(H_1(S_g,\Z/m)),\;\;\mbox{ for }m\geq 2.$$
They are denoted by $\GG(m)$ and called {\it abelian levels of order $m$}. The corresponding
level structures are then denoted by $\cM^{(m)}$. 

A result by Serre asserts that an automorphism of a smooth curve of genus $\geq 1$ acting 
trivially on its first homology group with $\Z/m$ coefficients, for $m\geq 3$, is trivial 
(cf. Lemma~2.9, Ch. XVII, \cite{A-C}). 
This implies that, for $g\geq 1$, any level structure $\cM^\ld$ over $\cM_{g,[n]}$, which dominates 
an abelian level structure $\cM^{(m)}$, with $m\geq 3$, is representable in the category of 
algebraic varieties.

\section{Looijenga level structures}\label{Loo}

There is another way to define levels of $\GG_{g,[n]}$ which turn out to be more
tractable in a series of applications. They were introduced by Looijenga in \cite{L1}.
Here, we generalize his definition and, in Section~\ref{symmetry}, we will give a 
geometric interpretation of the construction.

The idea, which underlies this construction and is more or less implicit in Looijenga's definition, 
is that the information present in the algebraic
fundamental group of an hyperbolic curve can be recovered from the first
homology groups of all its finite coverings with the induced Galois actions. This idea was further
developed by Mochizuki \cite{M} in his proof of the anabelian conjecture for hyperbolic curves.
Here, this idea is fully implemented to the study of Galois coverings of moduli spaces of curves. 

Let $K$ be a normal finite index subgroup of $\Pi_{g,n}$ and let $p_K\co S_K\ra S_{g,n}$ be 
the \'etale Galois covering with deck transformation group $G_K$, associated to such subgroup.  
There is then a natural monomorphism $G_K\hookra\GG(S_K)$ and
the quotient of the normalizer $N_{\GG(S_K)}(G_K)$ by $G_K$ identifies with
a finite index subgroup of $\GG_{g,[n]}$. If we assume, moreover, that $K$ is invariant for the
action of $\GG_{g,[n]}$, then any homeomorphism $f\co S_{g,n}\ra S_{g,n}$ lifts to a 
homeomorphism $\tilde{f}\co S_K\ra S_K$. So, there is a natural short exact sequence:
$$1\ra G_K\ra N_{\GG(S_K)}(G_K)\ra\GG_{g,[n]}\ra 1.\hspace{1cm}(1)$$

Let us assume that $K$ is a {\it proper} invariant subgroup of $\Pi_{g,n}$. Then, the covering
$p_K\co S_K\ra S_{g,n}$ ramifies non-trivially over all punctures of $S_{g,n}$. From Hurwitz's formula, 
for $g\geq 1$ or $n\geq 4$, it follows that the genus of the compact Riemann surface $\ol{S}_K$ 
obtained filling in the punctures of $S_K$ is at least one. For $m\geq 2$, let us then consider the 
natural representation $\rho_{(m)}\co\GG(S_K)\ra\Sp(H_1(\ol{S}_K,\Z/m))$.

\begin{remark}\label{no-hyperelliptic}
For $m\geq 3$ and $m=0$, the restriction of $\rho_{(m)}$ to $G_K$ is faithful. For $m=2$, the 
restriction $\rho_{(m)}|_{G_K}$ is faithful unless $G_K$ contains a hyperelliptic involution $\iota$
(cf. Lemma~2.9, Ch. XVII, \cite{A-C}), in which case $\iota$ generates the kernel of 
$\rho_{(m)}|_{G_K}$. In particular, for $g\geq 1$ or, for $g=0$, if $[\Pi_{0,n}:K]>2$, the group 
$G_K$ does not contain a hyperelliptic involution and so the restriction of $\rho_{(2)}$ to 
$G_K$ is also faithful.  
\end{remark}

Let us assume that the restriction $\rho_{(m)}|_{G_K}$ is faithful and let us denote by $G_K$  
its image. For $m\geq 2$, there is a natural representation:
$$\rho_{K,(m)}\co\GG_{g,[n]}\ra\left. N_{\Sp(H_1(\ol{S}_K,\Z/m))}(G_K)\right/G_K.$$
Let us denote the kernel of $\rho_{K,(m)}$ by $\GG^{K,(m)}$ 
and call it the {\it Looijenga level} associated to the subgroup $K$ of $\Pi_{g,n}$. 
The corresponding level structure is denoted by $\cM^{K,(m)}$.

Geometric levels can also be described in terms of the exact sequence $(1)$.
The geometric level $\GG^K$ associated to $K$ is indeed the set of elements of $\GG_{g,[n]}$ 
which admit a lift, through the natural epimorphism $N_{\GG(S_K)}(G_K)\tura\GG_{g,[n]}$,  to the 
centralizer $Z_{\GG(S_K)}(G_K)$. Thus, there is a short exact sequence:
$$1\ra Z(G_K)\ra Z_{\GG(S_K)}(G_K)\ra\GG^K\ra 1,\hspace{1cm}(2)$$
where $Z(G_K)$ is the center of the group $G_K$. If this center is trivial, the geometric level 
$\GG^K$ is then identified with the centralizer of $G_K$ inside $\GG(S_K)$.

Let us observe that a normal subgroup $H$ of $N_{\GG(S_K)}(G_K)$ such that
$H\cap G_K=\{1\}$ centralizes $G_K$, since, for $x\in H$ and $y\in G_K$, it holds 
$xyx^{-1}y^{-1}\in H\cap G_K$. Therefore, if $Z(G_K)=\{1\}$, then the centralizer
$Z_{\GG(S_K)}(G_K)$ is the maximal normal subgroup of the normalizer $N_{\GG(S_K)}(G_K)$
which intersects trivially the group $G_K$.

For a given finite index invariant subgroup $K\unlhd\Pi_{g,n}$, the subgroup $K':=[K,K]K^m$ of $K$, 
generated by commutators and $m$-th powers, is invariant and of finite index in $\Pi_{g,n}$ 
and the associated geometric level $\GG^{K'}$ is contained in the Looijenga level 
$\GG^{K,(m)}$. Conversely, it holds:

\begin{theorem}\label{comparison}For $2g-2+n>0$, let $K$ be a finite index invariant subgroup 
of $\Pi_{g,n}$. Then, for $m\geq 3$, or, for $m\geq 2$, if the quotient group $G_K$ does not contain 
a hyperelliptic involution (cf. Remark~\ref{no-hyperelliptic}), there is an inclusion of levels 
$\GG^{K,(m)}\unlhd\GG^K$.
\end{theorem}

\begin{proof}The action of the group $G_K$ on the surface $S_K$ 
induces, for $m\geq 2$, a faithful action on $H_1(\ol{S}_K,\Z/m))$.
Let us denote by $\tilde{\GG}^{K,(m)}$ the kernel of the natural representation
$$\tilde{\rho}_{K,(m)}\co N_{\GG(S_K)}(G_K)\ra\Sp(H_1(\ol{S}_K,\Z/m)).$$ 
It follows that $\tilde{\GG}^{K,(m)}\cap G_K=\{1\}$ 
and then $\tilde{\GG}^{K,(m)}$ is contained in the centralizer of $G_K$.
Hence, the Looijenga level $\GG^{K,(m)}$ is contained in the geometric level $\GG^K$.

\end{proof}

\begin{corollary}\label{cofinal}For any fixed integer $m\geq 2$, the set of Looijenga levels 
$\{\GG^{K,(m)}\}_{K\unlhd\Pi_{g,n}}$ forms an inverse system of finite index normal subgroups of 
$\GG_{g,[n]}$ which defines the same profinite topology as the tower of all geometric levels 
$\{\GG^K\}_{K\unlhd\Pi_{g,n}}$.
\end{corollary} 

Corollary~\ref{cofinal} vindicates the claim we made at the beginning of this section. 
For instance, in the study of the congruence topology of the Teichm\"uller group $\GG_{g,n}$,
Looijenga levels can replace  geometric levels. In Section~\ref{congruence}, the utility 
of this approach will emerge more clearly. 

In the hypotheses of Theorem~\ref{comparison}, the group $G_K$ acts faithfully on the homology 
group $H_1(\ol{S}_K,\Z/m)$. So, the intersection of $G_K$ with the level $\GG(m)$ of $\GG(S_K)$ 
is trivial and, by definition of the Looijenga level $\GG^{K,(m)}$, the natural epimorphism 
$N_{\GG(S_K)}(G_K)\tura\GG_{g,[n]}$ induces, for all $m\geq 2$, an isomorphism:
$$\tilde{\GG}^{K,(m)}=N_{\GG(S_K)}(G_K)\cap\GG(m)\cong\GG^{K,(m)}.\hspace{1cm}(3)$$
More explicitly, an $f\in\GG^{K,(m)}$ has a unique lift $\tilde{f}\co S_K\ra S_K$ which acts trivially 
on $H_1(\ol{S}_K,\Z/m)$. Hence, by Theorem~\ref{comparison}, for every $m\geq 2$, the 
Looijenga level $\GG^{K,(m)}$ has an associated Torelli representation:
$$t_{K,(m)}\co\GG^{K,(m)}\ra Z_{\Sp(H_1(\ol{S}_K,\Z))}(G_K).$$

A natural question is whether the image of $t_{K,(m)}$ has finite index in its codomain
for any invariant finite index subgroup $K$ of $\Pi_{g,n}$ as above, 
and $m\geq 2$. This question was addressed positively by Looijenga in \cite{L2}, for $n=0$
and the levels associated to the subgroup $\Pi^2$ (the so-called Prym levels).

\section{Level structures and loci of curves with symmetry}\label{symmetry}

In this section, we give a geometric construction which describes Looijenga level structures in 
terms of moduli of curves with a given group of automorphisms and of
abelian level structures. This has some immediate applications to the problem of describing locally
the D--M compactification of a Looijenga level structure.

As above, let $K$ be an invariant finite index subgroup of $\Pi_{g,n}$ and let 
$p_K\co S_K\ra S_{g,n}$ be the \'etale Galois covering with deck transformation group $G_K$, 
associated to such subgroup. Let $C$ be a smooth $n$-punctured, genus $g$ curve. 
Since the subgroup $K$ is invariant for the action of $\GG_{g,[n]}$, it determines the same finite 
index subgroup of $\pi_1(C)$ for any given marking $S_{g,n}\ra C$. Let then $C_K\ra C$ be the 
corresponding Galois covering.

A marking $\phi\co S_K\sr{\sim}{\ra} C_K$ identifies the group of automorphims 
$\aut(C_K)$ of the curve $C_K$ with a finite subgroup of $\GG(S_K)$. In this way,  
the Galois group of the covering $C_K\ra C$ is identified with some conjugate of $G_K$ in $\GG(S_K)$.

The theory of Riemann surfaces with symmetry (cf. \cite{GD-H}) describes the locus of 
$\cM(S_K)$, parametrizing curves which have a group of automorphisms conjugated to $G_K$  
as an irreducible closed substack $\cM_{G_K}$ of $\cM(S_K)$ with at most normal crossing 
singularities, whose normalization $\cM'_{G_K}$ is a smooth $G_K$-gerbe over the moduli stack 
$\cM_{g,[n]}$. In particular, there is a natural short exact sequence:
$$1\ra G_K\ra\pi_1(\cM'_{G_K})\ra\GG_{g,[n]}\ra 1.$$

A connected and analytically irreducible component of the inverse image of $\cM_{G_K}$, via 
the covering map $T(S_K)\ra\cM(S_K)$, is obtained as the fixed point set $T_{G_K}$ for the 
action of the subgroup $G_K<\GG(S_K)$ on the Teichm\"uller space $T(S_K)$. 

The submanifold $T_{G_K}$ of the Teichm\"uller space $T(S_K)$ is described as the set 
$(D,\phi)$ of Teichm\"uller points of $T(S_K)$ such that the group of automorphisms of the curve 
$D$ contains a subgroup which is topologically conjugated to $G_K$ by means of the 
homeomorphism $\phi\co S_K\sr{\sim}{\ra} D$ (cf. \cite{GD-H}, Theorem~A and B).
There is then a natural isomorphism of complex manifolds 
$T_{G_K}\cong T_{g,n}$. From this description, it follows that there is an isomorphism
$$N_{\GG(S_K)}(G_K)\cong\pi_1(\cM'_{G_K})$$
and then we recover the short exact sequence $(1)$ of Section~\ref{Loo}. 

The short exact sequence $(2)$ of Section~\ref{Loo} and the subsequent remarks have a geometric 
interpretation as well. If the center of $G_K$ is trivial, then the geometric level structure 
$\cM^K\ra\cM_{g,[n]}$ is the universal
trivializing covering for the gerbe $\cM'_{G_K}\ra\cM_{g,[n]}$, i.e. the pull-back of this gerbe along
a morphism $X\ra\cM_{g,[n]}$ is the trivial $G_K$-gerbe over $X$ if and only if this morphism 
factors through the geometric level structure $\cM^K\ra\cM_{g,[n]}$.

In the hypotheses of Theorem~\ref{comparison}, 
the group $G_K$ acts faithfully on the homology group $H_1(\ol{S}_K,\Z/m)$. Therefore, 
independently from the triviality of the center of $G_K$, the pull-back of the gerbe 
$\cM'_{G_K}\ra\cM_{g,[n]}$ along the Looijenga level structure $\cM^{K,(m)}\ra\cM_{g,[n]}$
is the trivial $G_K$-gerbe over $\cM^{K,(m)}$. 

In conclusion, the isomorphism $(3)$ of Section~\ref{Loo} yields a simple geometric interpretation 
of Loojenga level structures. It follows from $(3)$ that the Looijenga level structure $\cM^{K,(m)}$ 
is isomorphic to any of the irreducible components of the pull-back of the abelian level structure
$\cM(S_K)^{(m)}\ra\cM(S_K)$ over the natural morphism $\cM'_{G_K}\ra\cM(S_K)$:

\[ \begin{array}{ccc}
\hspace{1cm}\cM^{K,(m)}\; &\lra &\hspace{0,5cm}\cM(S_K)^{(m)}\\
\;\Big\da  &{\st\square} &\Big\da  \\
\hspace{0,5cm}\cM'_{G_K} &\lra & \cM(S_K).
\end{array} \]

The substack $\cM_{G_K}$ of $\cM(S_K)$ has normal crossing singularities if and only if the following 
occurs. Let $C_K$ and an embedding of $\aut(C_K)$ in $\GG(S_K)$ be defined as above. Then, the 
finite subgroup $\aut(C_K)<\GG(S_K)$ may contain, besides $G_K$, a $\GG(S_K)$-conjugate of 
$G_K$ distinct from $G_K$. If this happens, the stack $\cM_{G_K}$ self-intersects transversally inside 
the moduli stack $\cM(S_K)$, in the point parametrizing $C_K$. 
The situation is different for level structures dominating an abelian level.

\begin{proposition}\label{loogeo}\begin{enumerate}
\item For $m\geq 3$, an irreducible component of the inverse image 
of the closed substack $\cM_{G_K}$ of $\cM(S_K)$, in the abelian level structure $\cM(S_K)^{(m)}$, 
is smooth and isomorphic to the Looijenga level structure $\cM^{K,(m)}$ over $\cM_{g,[n]}$. 
\item For $m=2$, a self-intersection of such irreducible component may occur only in the locus 
parametrizing hyperelliptic curves and, if the group $G_K$ does not contain a hyperelliptic involution, 
its normalization is isomorphic to the Looijenga level structure $\cM^{K,(2)}$ over $\cM_{g,[n]}$.
\end{enumerate}
\end{proposition}

\begin{proof}The image of $T_{G_K}$ in the abelian level structure $\cM(S_K)^{(m)}$ has 
self-intersections in the points parameterizing the curve $C_K$, if and only if its automorphism 
group $\aut(C_K)$ contains two distinct $\GG(m)$-conjugates of $G_K$.

For $m\geq 3$ and, in case $C_K$ is not
hyperelliptic, for $m=2$ as well, the restriction of the natural representation 
$\rho_{(m)}\co\GG(S_K)\ra\Sp(H_1(\ol{S}_K,\Z/m)$ to $\aut(C_K)$ is faithful.

Let then $f$ be an element of the abelian level $\ker\rho_{(m)}$ of $\GG(S_K)$ such that both 
$G_K$ and $fG_Kf^{-1}$ are contained in the finite subgroup $\aut(C_K)$. It holds: 
$$\rho_{(m)}(fG_Kf^{-1})=\rho_{(m)}(f)\rho_{(m)}(G_K)\rho_{(m)}(f^{-1})=\rho_{(m)}(G_K).$$
Therefore, it follows that $fG_Kf^{-1}=G_K$. 

\end{proof}

The most obvious way to compactify a level structure $\cM^{\ld}$ over $\cM_{g,[n]}$ is to take 
the normalization $\ccM^{\ld}$ of $\ccM_{g,[n]}$ in the function field of $\cM^{\ld}$.
This definition can be formulated more functorially in the category of regular log schemes as done
by Mochizuki in \cite{Moch}. It is easy to see that the natural morphism of logarithmic stacks
$(\ccM^{\ld})^{log}\ra\ccM_{g,[n]}^{log}$ is log-\'etale, where $(\_)^{log}$ denotes
the logarithmic structure associated to the respective D--M boundaries 
$\dd\cM^\ld:=\ccM^\ld \ssm\cM^\ld$ and $\dd\cM_{g,[n]} :=\ccM_{g,[n]} \ssm\cM_{g,[n]}$. 
Viceversa, by the log purity Theorem, any finite, connected, log {\'e}tale Galois covering of
$\ccM_{g,[n]}^{log}$ is of the above type.

A basic property of (compactified) level structures is the following:
 
\begin{proposition}\label{Deligne}If a level $\GG^\lambda$ is contained in an abelian level
of order $m$, for some $m\ge 3$, then the level structure $\ccM^{\ld}_{g,n}$ is
represented by a projective variety.
\end{proposition}

\begin{proof}Even though this result is well known (cf. (ii), Proposition~2 \cite{Br}), its proof is rather 
technical and we prefer to give a sketch from the point of view of Teichm\"uller theory whose ideas 
will be useful later.

For a given stable $n$-pointed, genus $g$ curve $C$, let $\cN$ be its singular set and $\mathcal P$ its
set of marked points. {\it A degenerate marking} $\phi\co S_{g,n}\ra C$ is a continuous 
map such that $C\ssm\Im\,\phi={\mathcal P}$, the inverse image $\phi^{-1}(x)$, for all $x\in\cN$, is a simple 
closed curve (briefly, s.c.c.) on $S_{g,n}$ and the restriction of the marking 
$\phi\co S_{g,n}\ssm\phi^{-1}(\cN)\ra C\ssm(\cN\cup {\mathcal P})$ 
is a homeomorphism. {\it The Bers bordification} $\ol{T}_{g,n}$ of the Teichm\"uller space 
$T_{g,n}$ is the real analytic space which parametrizes pairs $(C,\phi)$, consisting 
of a stable $n$-pointed, genus $g$ curve $C$ and the homotopy class
of a degenerate marking $\phi\co S_{g,n}\ra C$
(cf. \S 3, Ch. II in \cite{Abikoff} for more details on this construction).

The natural action of $\GG_{g,[n]}$ on $T_{g,n}$ extends to $\ol{T}_{g,n}$. However, this action
is not anymore proper and discontinuous, since a boundary stratum has for inertia group the free
abelian group generated by the Dehn twists along the simple closed curves of $S_{g,n}$ which are
collapsed on such boundary stratum.

The geometric quotient $\ol{T}_{g,n}/\GG_{g,[n]}$ identifies with the real-analytic space underlying
the coarse moduli space $\ol{M}_{g,[n]}$ of stable $n$-pointed, genus $g$ curves.
Instead, the quotient stack $[\ol{T}_{g,n}/\GG_{g,[n]}]$ only admits a non-representable natural map 
to the D--M stack $\ccM_{g,[n]}$, because of the extra-inertia at infinity. 

We can assume that the given level $\GG^\ld$ is contained in the pure mapping class group
$\GG_{g,n}$. From the universal property of the normalization, it then follows that the level 
structure $\ccM^{\ld}$ is the relative moduli space of the morphism 
$[\ol{T}_{g,n}/\GG^\ld]\ra\ccM_{g,n}$. In order to prove that $\ccM^{\ld}$ is 
representable, we have to show that for all $x=(C,\phi)\in\dd T_{g,n}$, the stabilizer 
$\GG_x^\ld$ equals its normal subgroup $\GG_x^\ld\cap I_x$, where $I_x$ is the free abelian
subgroup of $\GG_{g,n}$ generated by the Dehn twists along the s.c.c.'s on $S_{g,n}$ which are 
contracted by the map $\phi$.

It is easy to see that $\GG_x/I_x$ is naturally isomorphic to the automorphism group of the
complex stable curve $C$ and that elements in $\GG(m)\cap\GG_x$ project to automorphisms 
acting trivially on $H_1(C,\Z/m)$. Therefore, the claim and then the proposition follows from the 
fact that, for $m\geq 3$, the only such automorphism is the identity  (cf. Lemma~4 \cite{Br}).

\end{proof}

\begin{remarks}\label{representability}\begin{enumerate}
\item It is interesting to notice that, for $n\geq 2$ and $m\geq 3$, 
the abelian level structure $\cM_{g,[n]}^{(m)}$ is representable while its D--M compactification
$\ccM_{g,[n]}^{(m)}$ is not, because stable curves with double pointed rational tails have a 
non-trivial automorphism which acts trivially on the homology. This shows that 
Proposition~\ref{Deligne} formulates a non-trivial property of level structures.

\item For $2g-2+n>0$, let $K$ be a finite index invariant subgroup of $\Pi_{g,n}$ with the properties 
that the quotient $\Pi_{g,n}/K$ surjects on $H_1(S_g,\Z/\ell)$, for some $\ell\geq 3$, and the 
inverse image of any non-peripheral s.c.c. $\gm$ on $S_{g,n}$, via the 
covering map $p_K\co S_K\ra S_{g,n}$, is a union of non-separating curves. Then, it holds 
$\GG^{K,(m)}\leq\GG_{g,n}$ and, by Theorem~\ref{comparison}, it holds as well
$\GG^{K,(m)}\leq\GG(\ell)$, for all $m\geq 2$. Therefore, it follows that the corresponding 
compactified Looijenga level structure $\ccM^{K,(m)}$ dominates $\ccM^{(\ell)}_{g,n}$, for 
$\ell\geq 3$, and then, by Proposition~\ref{Deligne}, is representable.
\end{enumerate}\end{remarks}

Returning to moduli of curves with symmetry, let us consider the Zariski closure 
$\ccM_{G_K}$ of the locus $\cM_{G_K}$ in the D--M compactification $\ccM(S_K)$.
The closed substack $\ccM_{G_K}$ consists of the points $[C]\in\ccM(S_K)$ such that the
automorphisms group of the curve $C$ contains a subgroup conjugated to $G_K$ via a
degenerate marking $\phi\co S_K\ra C$. 

This makes sense because an automorphism of $C$
preserves the sets $\cN$ and ${\mathcal P}$ of nodes and labels of $C$ and is determined by its restriction 
to $C\ssm(\cN\cup{\mathcal P})$, a self-homeomorphism of $S_K$ which preserves the closed submanifold 
$\phi^{-1}(\cN)$ is determined by its restriction to $S_K\ssm\phi^{-1}(\cN)$ and the restriction 
$\phi\co S_K\ssm\phi^{-1}(\cN)\ra C\ssm(\cN\cup{\mathcal P})$ is a homeomorphism.

By means of a degenerate marking $\phi\co S_{g,n}\ra C$, the finite group $\aut(C)$ then pulls back 
to a, not necessarily finite, subgroup $H_\phi$ of the mapping class group $\GG(S_K)$.

The closed substack $\ccM_{G_K}$ of the moduli stack $\ccM(S_K)$ has a self-intersection in the 
point parameterizing the curve $C$, if and only if, for some (and then for all) degenerate marking 
$\phi\co S_{g,n}\ra C$, the associated subgroup $H_\phi<\GG(S_K)$ contains two distinct 
$\GG(S_K)$-conjugates of $G_K$. 

The $G_K$-equivariant universal deformation of a stable curve endowed with a $G_K$-action
is formally smooth (cf. \S 5.1.1 \cite{B-R}). It follows that the stack $\ccM_{G_K}$ has at most normal 
crossing singularities and its normalization $\ccM_{G_K}'$ is a smooth $G_K$-gerbe over the moduli 
stack $\ccM_{g,[n]}$.

Let us describe the inverse image of the closed substack $\ccM_{G_K}$ in the compactified abelian 
level structure $\ccM(S_K)^{(m)}$.

The natural morphism of stacks $\ccM(S_K)^{(m)}\ra\ccM(S_K)$ is log-\'etale with respect to the 
logarithmic structures associated to the repsective D--M boundaries. By base-change via the 
natural morphism $\ccM'_{G_K}\ra\ccM(S_K)$, the morphism 
$\ccM'_{G_K}\times_{\ccM(S_K)}\ccM(S_K)^{(m)}\ra\ccM'_{G_K}$ is also log-\'etale 
with respect to the logarithmic structures associated to the respective D--M boundaries. 
From the log-purity Theorem, it then follows that the
connected components of the fiber product $\ccM'_{G_K}\times_{\ccM(S_K)}\ccM(S_K)^{(m)}$
are normal and hence isomorphic to the level structure $\ccM^{K,(m)}$.

Let us now describe locally the image of the level structure $\ccM^{K,(m)}$ in the compactified abelian 
level structure $\ccM(S_K)^{(m)}$.

\begin{lemma}\label{self-intersection} For $2g-2+n>0$, let $K$ be a proper invariant finite index 
subgroup of $\Pi_{g,n}$. Let $[C]\in\dd\cM_{G_K}$ be such that the natural representation
$\aut(C)\ra\mathrm{GL}(H_1(C,\Z/m))$ is injective for a given $m\geq 2$. Then, an irreducible 
component of the inverse image of $\ccM_{G_K}\subset\ccM(S_K)$ in the abelian level 
structure $\ccM(S_K)^{(m)}$ is normal in a neighborhood of a point parametrizing the curve $C$. 
\end{lemma}

\begin{proof}A connected and analytically irreducible component of the inverse image of 
$\ccM_{G_K}$ in the Bers bordification $\ol{T}(S_K)$ is given by the fixed point set $\ol{T}_{G_K}$ 
for the action of the subgroup $G_K$ of $\GG(S_K)$. Then, the natural isomorphism 
$T_{G_K}\cong T_{g,n}$ extends to an isomorphism of real-analytic spaces 
$\ol{T}_{G_K}\cong\ol{T}_{g,n}$.

A degenerate marking $\phi\co S_K\ra C$, for a point $[C]\in\dd\cM(S_K)$, determines a point 
$P\in\dd T(S_K)$. The stabilizer $\GG(S_K)_P$ of $P$, for the action of the Teichm\"uller group 
$\GG(S_K)$ on the Bers bordification $\ol{T}(S_K)$, is then described by the short exact 
sequence:
$$1\ra I_P\ra\GG(S_K)_P\ra\aut(C)\ra 1,$$
where $I_P$ is the free abelian group generated by the Dehn twists along the s.c.c.'s
on $S_K$ which are contracted by the marking $\phi$. 

A self-intersection of $\ccM_{G_K}$ occurs in the boundary point 
$[C]\in\dd\ccM_{G_K}$, if and only if the subgroup $\GG(S_K)_P$ of $\GG(S_K)$ contains 
two conjugates of $G_K$ which project to distinct subgroups of $\aut(C)$.
Similarly, the image of $\ol{T}_{G_K}$ in the abelian level structure $\ccM(S_K)^{(m)}$ has 
self-intersections in the points parameterizing the curve $C$, if and only if $\GG(S_K)_P$ contains 
two $\GG(m)$-conjugates of $G_K$ which project to distinct subgroups of $\aut(C)$.

If, for the given $m\geq 2$, the natural representation $\aut(C)\ra\mathrm{GL}(H_1(C,\Z/m))$ is 
injective, then two finite subgroups of $\GG(S_K)_P$, which differ by conjugation by an element of 
the abelian level $\GG(m)$, project to the same subgroup of $\aut(C)$. This proves the lemma.

\end{proof}

Let us then show that, under suitable hypotheses, the compactified Looijenga level structure
$\ccM^{K,(m)}$ fits in the commutative diagram:

\[ \begin{array}{ccc}
\hspace{1cm}\ccM^{K,(m)}\; &\hookra &\hspace{0,5cm}\ccM(S_K)^{(m)}\\
\;\Big\da  &{\st\square} &\Big\da  \\
\hspace{0,5cm}\ccM'_{G_K} &\ra & \ccM(S_K).
\end{array} \]

\begin{theorem}\label{normal}For $2g-2+n>0$, let $K$ be a finite index invariant subgroup of 
$\Pi_{g,n}$ satisfying the hypotheses of ii.) Remarks~\ref{representability}. Then, the closed 
substack $\ccM_{G_K}$ of $\ccM(S_K)$ does not meet the hyperelliptic locus
and, for $m\geq 2$, the associated Looijenga level structure $\ccM^{K,(m)}$ over $\ccM_{g,[n]}$ is 
representable and isomorphic to any of the irreducible components of the inverse image of 
$\ccM_{G_K}$ in the abelian level structure $\ccM(S_K)^{(m)}$. 
\end{theorem}

\begin{proof}By ii.) Remarks~\ref{representability}, the level structure $\ccM^{K,(m)}$ is representable 
for all $m\geq 2$. In particular, the image of the finite morphism 
$\ccM'_{G_K}\times_{\ccM(S_K)}\ccM(S_K)^{(2)}\ra\ccM(S_K)^{(2)}$ does not meet 
the hyperelliptic locus of $\ccM(S_K)^{(2)}$, which has a generic non-trivial automorphism.
But then $\ccM_{G_K}$ does not meet the hyperelliptic locus of $\ccM(S_K)$ either.

If the hypotheses of ii.) Remarks~\ref{representability}, are satisfied, from the proof of
Proposition~\ref{Deligne}, it follows that, for all $m\geq 2$ and all points
$[C]\in\dd\ccM_{G_K}\subset \ccM(S_K)$, the natural representation 
$\aut(C)\ra\mathrm{GL}(H_1(C,\Z/m))$ is injective. Therefore, by Lemma~\ref{self-intersection}
and Proposition~\ref{loogeo}, for $m\geq 2$, the irreducible components of the inverse 
image of the closed substack $\ccM_{G_K}$ of $\ccM(S_K)$ in the abelian level structure 
$\ccM(S_K)^{(m)}$ have no self-intersections and so they are isomorphic to $\ccM^{K,(m)}$. 

\end{proof} 

\begin{remark}\label{smoothness criterion}
From the above description of a Looijenga level structure $\ccM^{K,(m)}$, for $m\geq 2$, we can
derive a simple smoothness criterion:\,
if the inverse image of the closed substack $\ccM_{G_K}$ of $\ccM(S_K)$ in the abelian level 
structure $\ccM(S_K)^{(m)}$ avoids its singular locus, then the compactified 
Looijenga level structure  $\ccM^{K,(m)}$ is smooth. This easily follows from the fact that the closed 
substack $\ccM_{G_K}$ meets transversally the branch locus of the covering
$\ccM(S_K)^{(m)}\ra\ccM(S_K)$, which is contained in the D--M boundary of $\ccM(S_K)$.
\end{remark}
   
There is a very elementary and effective method to describe the compactifications $\ccM^\ld$, 
locally in the analytic topology. The natural morphism $\ccM^\ld\ra\ccM_{g,[n]}$ is \'etale
outside the D--M boundary. Therefore, we need just to consider the case of a point  $x\in\dd\cM^\ld$.

Let us observe first that, since the natural morphism of stacks $\ccM_{g,n}\ra\ccM_{g,[n]}$
is \'etale, it is not restrictive to consider only level structures $\ccM^\ld$ over $\ccM_{g,n}$.

Let then $\cB\ra\ccM_{g,n}$ be an analytic neighborhood of the image $y$ of $x$ in $\ccM_{g,n}$ 
such that:
\begin{itemize} 
\item local coordinates $z_{\sst 1},...,z_{\sst 3g-3+n}$ embeds $\cB$ in $\C^{3g-3+n}$ 
as an open ball;
\item the curve $C:=\pi^{-1}(y)$ is maximally degenerate inside the pull-back $\cC\stackrel{\pi}{\ra}\cB$
of the universal family over $\cB$; 
\item  an {\'e}tale groupoid representing $\ccM_{g,n}$ trivializes over $\cB$ to 
$\mbox{Aut}(C)\times\cB\rightrightarrows\cB$. 
\end{itemize}
Let $\{Q_1,\dots,Q_s\}$ be the set of singular points of $C$ and let
$z_i$, for $i=1,\dots ,s$, parametrize curves where the singularity $Q_i$
subsists. The discriminant locus $\dd\cB\subset\cB$ of $\pi$
has then equation $z_{\sst 1}\cdots z_s=0$. Let $U=\cB\ssm\dd\cB$ and suppose that $a\in U$. 
The natural morphism $U\ra\cM_{g,n}$ induces a homomorphism of fundamental groups:
$$\psi_U\co\pi_1(U,a)\ra\pi_1(\cM_{g,n},a)\equiv\GG_{g,n}.$$ 
A connected component $U^{\ld}$ of $U\times_{\ccM_{g,n}}\ccM^{\ld}$ is then
determined by the subgroup $\psi_U^{-1}(\GG^\ld)$ of the group $\pi_1(U,a)$.
    
The neighborhood $U$ is homotopic to the $s$--dimensional torus $(S^1)^s$. 
Therefore, the fundamental group $\pi_1(U,a)$ is abelian and freely generated by simple loops 
$\gamma_i$ around the divisors $z_i$, for $i=1,\dots,s$. 

Since we have fixed a homeomorphism $S_{g,n}\sr{\sim}{\ra}\cC_a$, the loops $\gm_i$, 
for $i=1,\dots,s$, can be lifted to disjoint, not mutually homotopic loops $\tilde{\gamma}_i$ 
on $S_{g,n}$, which do not bound a disc with less than $2$ punctures. The loop 
$\gamma_i$ is mapped by $\psi_U$ exactly in the Dehn twist $\tau_{\tilde{\gamma}_i}$ of 
$\GG_{g,n}$, for $i=1,\dots,s$. In particular, the representation $\psi_U$ is faithful. 

Let $E_{\Sigma(C)}$ be the free abelian group generated by the edges of the dual graph $\Sigma(C)$ 
of the stable curve $C$. The edges of the dual graph correspond to the isotopy classes of the s.c.c. 
$\gm_e$ in $S_{g,n}$ which become isotrivial specializing to $C$. Let us then identify the group 
$E_{\Sigma(C)}$ both with the free abelian group generated in $\GG_{g,n}$ by the set of Dehn twists 
$\{\tau_{\gm_e}\}$ and with the fundamental group $\pi_1(U,a)$.

For a given level $\GG^\ld$, let us denote by $\psi_U^\ld$ the composition of $\psi_U$ with the
natural epimorphism $\GG_{g,n}\ra\GG_{g,n}/\GG^\ld$. Hence, the kernel of $\psi_U^\ld$
is identified with $E_{\Sigma(C)}\cap\GG^\ld$, which we call the local monodromy kernel in $U$
of the level $\GG^\ld$.

The local monodromy kernels for the abelian levels $\GG(m)$ of $\GG_{g,n}$ are easy to compute.
Let $N_{\Sigma(C)}$ and $S_{\Sigma(C)}$ be respectively the subgroups of 
$E_{\Sigma(C)}$ generated by edges corresponding to  non-separating s.c.c. and by edges 
corresponding to separating s.c.c.. 
Let instead $P_{\Sigma(C)}$ be the subgroup generated  by elements of the form $e_1-e_2$, where 
$\{e_1,e_2\}$ corresponds to a cut pair on $C$. 

\begin{theorem}\label{monodromy coefficients}With the above notations, 
the kernel of $\psi_U^{(m)}$ is given by:
$$m\, N_{\Sg(C)}+P_{\Sigma(C)}+ S_{\Sigma(C)}.$$
Therefore, the singular locus of the abelian level structure $\ccM_{g,[n]}^{(m)}$, 
for $m\geq 2$, is contained in the strata parametrized by cut pairs on $S_{g,n}$.
\end{theorem}

\begin{proof}Let $\sg:=\{\gm_0,\ldots,\gm_k\}$ be a set of disjoint, non-mutually homotopic 
non-peripheral s.c.c.'s on $S_{g,n}$ and let us prove that 
$\tau_{\gm_0}^{h_0}\cdot\ldots\cdot\tau_{\gm_k}^{h_k}\in\GG(m)$ if and only if it maps
to the identity in the quotient of the free abelian group $E$ generated by the set of 
Dehn twists $\{\tau_\gm\}_{\gm\in\sg}$ by its subgroup $N+P+S$ generated by $m$-th powers of
Dehn twists $\tau_\gm$, for $\gm\in\sg$, bounding pair maps $\tau_\gm\tau_{\gm'}^{-1}$, for 
$\gm,\gm'\in\sg$ a cut pair, and Dehn twists $\tau_\gm$, for $\gm\in\sg$ a separating s.c.c..

Choosing an adapted system of generators for the homology of the compact Riemann surface 
$S_g$, it is easy to see that $N+P+S<\GG(m)$. So let us assume that 
$\tau_{\gm_0}^{h_0}\cdot\ldots\cdot\tau_{\gm_k}^{h_k}\in\GG(m)$ and let us show that this element
maps to the identity in the quotient $E/(N+P+S)$.

Replacing $\tau_{\gm_0}^{h_0}\cdot\ldots\cdot\tau_{\gm_k}^{h_k}$ by an element which is 
congruent modulo $N+P+S$, we can assume that the set $\sg$ does not contain cut pairs or
separating s.c.c.'s and that it holds $0\leq h_i<m$, for all $i=0,\ldots,k$.

From a simple topological argument, it then follows that, for two given distinct s.c.c.'s $\gm_i$ and 
$\gm_j$ in the set $\sg$, there is a non-separating s.c.c.
$\gm$ on $S_{g,n}$ which is disjoint from the other s.c.c.'s in $\sg$ and intersects $\gm_i$ and 
$\gm_j$ transversally in one point. For a given s.c.c. $\alpha$ on $S_{g,n}$, let us denote 
by $\ol{\alpha}\in H_1(S_g,\Z/m)$ the cycle associated to a given orientation of $\alpha$. 
For suitable orientations of $\gm$, $\gm_i$ and $\gm_j$, it holds:
$$\tau_{\gm_0}^{h_0}\cdot\ldots\cdot\tau_{\gm_k}^{h_k}(\ol{\gm})=\ol{\gm}+\ol{\gm_i}^{h_i}+
\ol{\gm_j}^{h_j}\in H_1(S_g,\Z/m).$$
Since the cycles $\ol{\gm_i}$ and $\ol{\gm_j}$ are primitive and linearly independent in the
homology group $H_1(S_g,\Z/m)$, the above identity implies that $h_i=h_j=0$.

\end{proof}

Let $K$ be a finite index invariant subgroup of $\Pi_{g,n}$, satisfying the hypotheses 
of ii.) Remarks~\ref{representability}, and let $\ccM^{K,(m)}$ be the associated Looijenga level 
structure over $\ccM_{g,[n]}$, which then is representable and dominates the abelian level 
$\ccM_{g,n}^{(\ell)}$, for some $\ell\geq 3$. 

Let us denote by $\psi\co\ccM(S_K)^{(m)}\ra\ccM(S_K)$ the natural finite morphism. 
In Theorem~\ref{normal}, for $m\geq 2$, we identified the level structure $\ccM^{K,(m)}$ 
with an irreducible component of the inverse image $\psi^{-1}(\ccM_{G_K})$, 
which is a normal and proper subvariety of $\ccM(S_K)^{(m)}$.

It is easy to determine the boundary strata of $\ccM(S_K)^{(m)}$ which are met by the inverse 
image $\psi^{-1}(\ccM_{G_K})$, in terms of the covering $p_K\co S_K\ra S_{g,n}$ and systems of
s.c.c.'s on $S_{g,n}$. 

\begin{definition}\label{admissible}
We say that a set $\sg=\{\gm_0,\ldots,\gm_k\}$ of disjoint, non-mutually homotopic 
non-peripheral s.c.c.'s on $S_{g,n}$ is {\it admissible}. Let us also denote by the same letter 
$\sg$ their union in $S_{g,n}$. 
\end{definition}

For an admissible set $\sg$ of s.c.c.'s on $S_{g,n}$,
the inverse image $p_K^{-1}(\sg)$ is also an admissible set of s.c.c.'s on $S_K$ and so 
determines a closed boundary stratum $B_{p_K^{-1}(\sg)}^{(m)}$ of $\ccM(S_K)^{(m)}$ 
(cf. Section~\ref{simplicial} for details on this correspondence). It is clear that:
$$\psi^{-1}(\ccM_{G_K})\cap\dd\ccM(S_K)^{(m)}\subset
\bigcup_{\sg\subset S_{g,n}} B_{p_K^{-1}(\sg)}^{(m)}.$$

Therefore, from Remark~\ref{smoothness criterion} and Theorem~\ref{monodromy coefficients}, 
it follows:

\begin{theorem}\label{smoothness criterion II}For $2g-2+n>0$, let $K$ be a finite index invariant 
subgroup of $\Pi_{g,n}$ such that the covering $p_K\co S_K\ra S_{g,n}$ has the property that, for all 
admissible sets $\sg$ of s.c.c.'s, the inverse image $p_K^{-1}(\sg)$ does not contain cut pairs. 
Then, the compact Looijenga level structure $\ccM^{K,(m)}$ over $\ccM_{g,[n]}$ is smooth for all 
$m\geq 2$.
\end{theorem}

It is now possible both to clarify and improve substantially the main result of \cite{L1}.
The only technical result needed from \cite{L1} is Proposition~2, which states that, for $g\geq 2$ 
and $K=\Pi_g^2$, the covering $p_K\co S_K\ra S_g$ 
is such that, for any admissible set $\sg$ of s.c.c.'s on $S_{g}$, the inverse image 
$p_K^{-1}(\sg)$ does not contain cut pairs. It is not difficult to extend the argument to
invariant subgroups of $\Pi_{g,n}$ satisfying more general hypotheses: 

\begin{lemma}\label{no cut-pairs}For $2g-2+n>0$, let $K$ be a finite index invariant subgroup of 
$\Pi_{g,n}$ such that the covering map $p_K\co S_K\ra S_{g,n}$ ramifies non-trivially
over all punctures of $S_{g,n}$ and let $K_\ell:=[K,K]K^\ell$, for an integer $\ell\geq 2$.
This is also a finite index invariant subgroup of $\Pi_{g,n}$ and
the associated covering map $p_{K_\ell}\co S_{K_\ell}\ra S_{g,n}$ is such that, for an admissible 
set $\sg$ of s.c.c.'s on $S_{g,n}$,  the inverse image $p_{K_\ell}^{-1}(\sg)$ 
contains no separating curves or cut pairs. 
\end{lemma}
\begin{proof}The condition on the covering map $p_K\co S_K\ra S_{g,n}$ implies, by Hurwitz 
Theorem, that, for an admissible set $\sg$ of s.c.c.'s on $S_{g,n}$, the inverse image
$p_{K}^{-1}(\sg)$ determines an admissible set of s.c.c.'s on the closed Riemann surface $\ol{S}_K$
and that this is of genus at least $2$ (except, possibly, for the case $g=0$ and $n=3$, 
which is irrelevant).

Let us denote by $\phi_\ell\co S_{K_\ell}\ra S_K$ the abelian covering defined by the subgroup 
$K_\ell$ of $K$. It is enough to prove that the dual graph $\Sg_{K_\ell}$ of the stable curve obtained 
collapsing on $S_{K_\ell}$ the s.c.c.'s contained in $\phi_\ell^{-1}(\sg)$ stays connected if we remove 
two edges, where $\sg$ is an admissible set of s.c.c.'s on the closed Riemann surface $\ol{S}_K$.
Let $\Sg_K$ be the dual graph of the stable curve obtained collapsing on $S_K$ the s.c.c.'s in $\sg$.
As in the proof of Proposition~2 \cite{L1}, it is enough to consider the case in which 
the cardinality of $\sg$ is  $\leq 2$.

Let us observe in the first place that the fact that Dehn twists on $S_K$ lift to
homeomorphisms of the covering surface $S_{K_\ell}$ implies that each s.c.c. contained in 
$\phi_\ell^{-1}(\sg)$ bounds two distinct connected components of $S_K$. Then, the graph 
$\Sg_{K_\ell}$ does not contain loops and there is a natural action without inversions of 
the deck transformation group $H_1(K,\Z/\ell)$ on $\Sg_{K_\ell}$ with quotient the graph $\Sg_K$.

A vertex $v$ of the graph $\Sg_K$ corresponds to a connected component $S'$ of $S_K\ssm\sg$ 
The vertices of $\Sg_{K_\ell}$ lying over $v$ have the same stabilizer for the action of $H_1(K,\Z/\ell)$
and this is the image of $H_1(S')$ in $H_1(K,\Z/\ell)$ for the homomorphism induced by inclusion.
Similarly, the edges of $\Sg_{K_\ell}$ lying over an edge $e$ of $\Sg_K$, corresponding to the
s.c.c. $\gm$ of the set $\sg$, have for stabilizer the associated cyclic subgroup $I_e$ of 
$H_1(K,\Z/\ell)$.

The number of vertices (resp. edges) of $\Sg_{K_\ell}$ lying over $v$ (resp. $e$) is given by the 
index of the stabilizer of one of these vertices (resp. edges) in the group $H_1(K,\Z/\ell)$. 

A simple count and symmetry show that each vertex of $\Sg_{K_\ell}$ is connected to 
the rest of the graph by at least three edges and that, when two vertices are connected by an edge, 
they are connected by at least two edges. Therefore, the graph cannot be disconnected removing 
two edges.

\end{proof} 

Looijenga's results then generalize to all pairs $(g,n)$, for $2g-2+n>0$: 

\begin{theorem}\label{smooth covers} For $2g-2+n>0$, let $K_\ell$ be a
subgroup of $\Pi_{g,n}$ of the type defined in Lemma~\ref{no cut-pairs}. Then, the associated 
compact Looijenga level structure $\ccM^{K_\ell,(m)}$ over $\ccM_{g,[n]}$ is smooth for all 
$m\geq 2$ and representable for $\ell\geq 3$ or $m\geq 3$. 
\end{theorem}
\begin{proof}The statement about smoothness is a straightforward consequence of 
Theorem~\ref{smoothness criterion II} and Lemma~\ref{no cut-pairs}.

The hypotheses of Lemma~\ref{no cut-pairs} on the subgroup $K$ also imply 
that $\GG^{K_\ell,(m)}<\GG_{g,n}$, for all $m\geq 2$. Moreover,
the quotient $\Pi_{g,n}/K_\ell$ surjects onto $H_1(S_g,\Z/\ell)$. Therefore, if $\ell\geq 3$ or 
$m\geq 3$, it holds $\GG^{K_\ell,(m)}<\GG_{g,n}\cap\GG(k)$, for some $k\geq 3$, and the 
associated compact level structure is representable.

\end{proof}

It is not difficult to determine explicitly the local monodromy coefficients for the levels 
$\GG^{K,(m)}$ satisfying the hypotheses of Theorem~\ref{smoothness criterion II}.

For a non-peripheral s.c.c. $\gm$ on $S_{g,n}$, let $c_\gm$ be the order in the quotient group 
$G_K=\Pi_{g,n}/K$ of an element of $\Pi_{g,n}$, whose free homotopy class contains 
$\gm$. By the invariance of $K$, the positive integer $c_\gm$ depends 
only on the topological type of $S_{g,n}\ssm\gm$. It then holds:

\begin{proposition}\label{monodromy coefficients2} Let $K$ be a finite index invariant subgroup 
of $\Pi_{g,n}$ satisfying the hypotheses of Theorem~\ref{smoothness criterion II}. 
With the same notations of Theorem~\ref{monodromy coefficients}, 
for $m\geq 2$, the kernel of $\psi_U^{K,(m)}$ is given by:
$$\sum_{e\in\Sg(C)}mc_{\gm_e}e.$$
\end{proposition}

\begin{proof}For a non-peripheral s.c.c. $\gm$ on $S_{g,n}$, let us denote by $k_\gm$ the order of 
the image of the Dehn twist $\tau_\gm$ in the quotient group $\GG_{g,n}/\GG^{K,(m)}$ and let
$\amalg_{i=1}^h S_i:=\ol{S}_{K}\ssm p_{K}^{-1}(\gm)$. 

Let us show first that $c_\gm$ divides $k_\gm$.
This follows from the fact that any lift to $\ol{S}_{K}$ of a power 
$\tau_\gm^k$,  where $c_\gm$ does not divide $k$, either switches 
the connected components of $\ol{S}_{K}\ssm p_{K}^{-1}(\gm)$, and does not act trivially on 
$H_1(\ol{S}_{K},\Z/m)$ for $m\geq 2$, or restricts to a non-trivial finite order homeomorphism of 
some connected component $S_i$ of $\ol{S}_{K}\ssm p_{K}^{-1}(\gm)$, which then acts 
non-trivially on the image of $H_1(S_i,\Z/m)$ in $H_1(\ol{S}_{K},\Z/m)$, for $m\geq 2$.

The power  $\tau_\gm^{c_\gm}\in\GG_{g,[n]}$ of the given Dehn twist 
lifts to a product $\xi_\gm:=\prod_{\delta\in\sg_\gm}\tau_\delta\in\GG(S_{K})$ of Dehn twists along the 
set $\sg_\gm$ of disjoint s.c.c.'s on $S_{K}$ contained in $p_{K}^{-1}(\gm)$. This lift is contained in 
the centralizer of $G_{K}$. Therefore, another lift is of the form $\xi_\gm\cdot\alpha$, for some 
$\alpha\in G_K$, where $\xi_\gm$ and $\alpha$ are commuting elements of $\GG(S_{K})$. 

Any power of $\xi_\gm$ fixes the subsurface $S_i$ of $\ol{S}_{K}$ and acts trivially on the image of 
$H_1(S_i,\Z/m)$ in $H_1(\ol{S}_{K},\Z/m)$, for $i=1,\ldots,h$, while no non-trivial element of $G_K$ 
has this property. Hence, the cyclic subgroup generated by the image of $\xi_\gm$ in the group 
$\Sp(H_1(\ol{S}_{K},\Z/m)$ intersects trivially the image of the group $G_K$.

It then follows that the image of $\xi_\gm$ in $\Sp(H_1(\ol{S}_{K},\Z/m)$ is of minimal order 
between those of all possible lifts of $\tau_\gm^{c_\gm}$ to $\GG(S_{K})$.

By hypothesis, the set $\sg_\gm$ consists of non-separating s.c.c. and does not contain cut pairs.
Hence, by Theorem~\ref{monodromy coefficients}, the image of $\xi_\gm$ in
$\Sp(H_1(\ol{S}_{K},\Z/m)$ has order $m$. Therefore, it holds $k_\gm=mc_\gm$. 

\end{proof}

\section[The boundary of level structures]
{The boundary of level structures and the complex of curves}\label{simplicial}
For simplicity, we restrict here to level structures over $\ccM_{g,n}$. Most of the result
of this section are well known or have already appeared elsewhere (\S 3 of \cite{PFT}) 
but there are also some rectifications to \cite{PFT} (notably, in 
Proposition~\ref{without rotations} and Proposition~\ref{no self-intersection}). 

For a point $[C]\in\dd\cM_{g,n}$, let $C_{g_1,n_1}\amalg\ldots\amalg C_{g_h,n_h}$
be the normalization of $C$, where $C_{g_i,n_i}$, for $i=1,\ldots, h$, is a genus $g_i$ smooth curve
with $n_i$ labels on it (the labels also include the inverse images of singularities in $C$).
Then, there is a natural morphism, which we call {\it boundary map}:
$$\beta_C:\ccM_{g_1,n_1}\times\dots\times\ccM_{g_h,n_h}\lra\ccM_{g,n}.$$ 
The image of $\cM_{g_1,n_1}\times\dots\times\cM_{g_h,n_h}$ by $\beta_C$ parametrizes curves
homeomorphic to $C$ and is called a {\it stratum}. We denote
the restriction of $\beta_C$ to $\cM_{g_1,n_1}\times\dots\times\cM_{g_h,n_h}$ by
$\dot{\beta}_C$ and call it a {\it stratum map}. In general, these morphisms are not
embeddings. 

A variant of the Bers bordification $\ol{T}_{g,n}$ of the Teichm\"uller space $T_{g,n}$ is the
{\it Harvey bordification} $\wh{T}_{g,n}$. This can be constructed by means of
the D--M compactification $\ccM_{g,n}$. 

Let $\wh{\cM}_{g,n}$ be the real oriented blow-up of $\ccM_{g,n}$ along the 
D--M boundary.  Its boundary $\dd\wM_{g,n}:=\wM_{g,n}\ssm\cM_{g,n}$ is 
homeomorphic to a deleted tubular neighborhood of the D--M boundary of $\ccM_{g,n}$ 
and the natural projection $\wh{\cM}_{g,n}\ra\ccM_{g,n}$ restricts over each codimension $k$ 
stratum to a bundle in $k$-dimensional tori. The inclusion $\cM_{g,n}\hookra\wh{\cM}_{g,n}$  is a 
homotopy equivalence and then induces an inclusion of the respective
universal covers $T_{g,n}\hookra\wh{T}_{g,n}$. 

From Proposition~\ref{Deligne}, it follows that $\wh{T}_{g,n}$ is
representable and, therefore, is a real analytic manifold with corners containing $T_{g,n}$ as an
open dense submanifold. The {\it ideal boundary} of Teichm{\"u}ller space is defined to be
$\dd\wh{T}_{g,n}:=\wh{T}_{g,n}\ssm T_{g,n}$.

The strata of $\dd\wT_{g,n}$ lying above the boundary map $\beta_C$, are isomorphic to the 
product $(\R^+)^k\times\wT_{g_1,n_1}\times\dots\times\wT_{g_h,n_h}$.
Thus, they are all contractible. Hence, $\dd\wT_{g,n}$ is homotopy equivalent to the geometric 
realization of the nerve of its cover by irreducible components.  

This nerve is described by 
the simplicial complex whose simplices consist of sets of distinct, non-trivial 
isotopy classes of non-peripheral s.c.c. on $S_{g,n}$, such that they admit a set of disjoint 
representatives. This is the complex of curves $C(S_{g,n})$ of 
$S_{g,n}$. It is easy to check that the combinatorial dimension of $C(S_{g,n})$ is one less the 
complex dimension of the moduli space 
$\cM_{g,n}$, i.e.:\, $n-4$ for $g=0$ and $3g-4+n$ for $g\geq 1$.
The natural action of $\GG_{g,n}$ on $\dd\wh{T}_{g,n}$ then induces a simplicial action 
of $\GG_{g,n}$ on $C(S_{g,n})$.

Let $\wh{\cM}_{g,n}^\ld$ be the oriented real blow-up of $\ccM_{g,n}^\ld$ along $\dd\cM^\ld$.
It holds $\wh{\cM}_{g,n}^\ld\cong [\wT_{g,n}/\GG^\ld]$ and then also
$\dd\wh{\cM}_{g,n}^\ld\cong [(\dd\wT_{g,n})/\GG^\ld]$. Therefore, the quotient 
$|C(S_{g,n})|/\GG^\ld$ is a cellular complex which realizes the nerve of the cover of 
$\dd\wh{\cM}^\ld$ by irreducible components and, thus, the nerve of the cover of $\dd\cM^\ld$ 
by irreducible components. 

\begin{definition}\label{nervedef} For $2g-2+n>0$, let $\GG^\ld$ be a level of $\GG_{g,n}$. Then, 
we define $C^\ld(S_{g,n})_\bt$ to be the finite simplicial set associated to the cellular complex
$|C(S_{g,n})|/\GG^\ld$. This simplicial set parameterizes the nerve of the cover of 
$\dd\cM^\ld$ by its irreducible components.
\end{definition}

Let $\sg$ be a non-degenerate $k$-simplex of $C(S_{g,n})_\bt^\ld$, such that, for some lift $\td{\sg}$
of $\sg$ to the curve complex $C(S_{g,n})$, it holds 
$S_{g,n}\ssm\td{\sg}:=S_{g_1,n_1}\amalg\ldots\amalg S_{g_h,n_h}$. Then, to the simplex $\sg$,
is associated a boundary map $\beta_{\sg}^\ld\co\delta_{\sg}^\ld\ra\ccM^\ld$ which is the restriction
to an irreducible component of the pull-back $\beta'$ of the boundary map $\beta$ associated to the 
topological type of the Riemann surface $S_{g,n}\ssm\td{\sg}$:

\[ \begin{array}{ccc}
X\; &\lra &\ccM_{g_1,n_1}\times\dots\times\ccM_{g_h,n_h}\\
\;\Big\da {\scriptstyle \beta'} &{\st\square} &\Big\da {\st \beta} \\
\ccM^{\ld} &\lra & \ccM_{g,n}.
\end{array} \]

The image of $\beta_{\sg}^\ld$ is the {\it closed boundary stratum associated to $\sg$}.
Similarly, we call the restriction $\dot{\beta}_\sg^\ld\co\od_\sg^\ld\ra\ccM_{g,n}^\ld$ of 
$\beta_\sg^\ld$ over $\cM_{g_1,n_1}\times\dots\times\cM_{g_h,n_h}$ the 
{\it stratum map of $\ccM^\ld$ associated to $\sg$} 
and its image is the {\it  boundary stratum associated to $\sg$}. By log-{\'e}tale base change, the
natural morphism $\od_\sg^\ld\ra\cM_{g_1,n_1}\times\dots\times\cM_{g_h,n_h}$ is {\'e}tale.

Similar definitions can be given with $\wM_{g,n}^\ld$ in place of $\ccM_{g,n}^\ld$. For a
simplex $\sg\in C(S_{g,n})$, we define the {\it ideal boundary map}
$\hat{\beta}_{\sg}^\ld\co\hd_{\sg}^\ld\ra\wM_{g,n}^\ld$, as the pull-back of $\beta_{\sg}^\ld$ via
the blow-up map $\wM^\ld_{g,n}\ra\ccM_{g,n}^\ld$.

The fundamental group of $\hd_\sg^\ld$ is described as follows. 
Let $\hat{\delta}_\sg^\ld$ be the real oriented blow-up of $\delta_\sg^\ld$ along the divisor
$\delta_\sg^\ld\ssm\od_\sg^\ld$. The embedding $\od^\ld_\sg\hookra\hat{\delta}_\sg^\ld$
is a homotopy equivalence and $\hd_\sg^\ld$ is a bundle over
$\hat{\delta}_\sg^\ld$ in $k$-dimensional tori. Hence, there is an isomorphism 
$\pi_1(\hat{\delta}_\sg^\ld)\cong\pi_1(\od_\sg^\ld)$ and a short exact sequence:
$$1\ra\bigoplus_{\gamma\in \sg}\Z\cdot\gm\ra
\pi_1(\hd^\ld_\sg)\ra\pi_1(\hat{\delta}_\sg^\ld)\ra 1.$$

A connected component of the fibred product $\hd_\sg\times_{\wM_{g,n}}\wT_{g,n}$ is naturally 
isomorphic to $(\R^+)^k\times\wT_{g_1,n_1}\times\dots\times\wT_{g_h,n_h}$. Therefore, the 
fundamental group of $\hd_\sg$ is isomorphic to the subgroup of elements in
$\GG_{g,n}$ which stabilize one of these connected components, preserving,
moreover, the order of its factors. So, if we let
$$\GG_{\vec{\sg}}:=\{f\in\GG_{g,n}|\; f(\vec{\gamma})=\vec{\gamma},\, \forall\gamma\in \sg\},$$
where $\vec{\gamma}$ is the s.c.c. $\gamma$ endowed with an orientation, it holds 
$\pi_1(\hd_\sg)\cong \GG_{\vec{\sg}}$ and, for the trivial level, the above short exact
sequence takes the more familiar form
$$1\ra\bigoplus_{\gamma\in \sg}\Z\cdot\tau_\gamma\ra\GG_{\vec{\sg}}
\ra\GG_{g_1,n_1}\times\dots\times\GG_{g_h,n_h}\ra 1.$$
By the same argument, more generally, it holds
$\pi_1(\hd^\ld_\sg)\cong\GG^\ld\cap \GG_{\vec{\sg}}$.

The fundamental group of $\hat{\beta}_{\sg}(\hd_{\sg})$ is
isomorphic to the stabilizer $\GG_\sg$ of $\sg$, for the action of $\GG_{g,n}$ 
on $C(S_{g,n})$ and more generally, it holds
$\pi_1(\hat{\beta}_{\sg}^\ld(\hd_{\sg}^\ld))\cong\GG^\ld\cap\GG_\sg$.

Let us observe that an $f\in\GG_\sg$ can switch the s.c.c. in $\sg$ as well as their orientations.
Hence, denoting by $\Sigma_{\sg}\{\pm\}$ the group of signed permutations of the set
$\sg$, there is an exact sequence 
$$1\ra\GG^\ld\cap \GG_{\vec{\sg}}\ra\GG^\ld\cap\GG_\sg\ra\Sigma_{\sg}\{\pm\}$$ 
and the Galois group of the \'etale covering $\od_\sg^\ld\ra\Im\,\dot{\beta}_\sg^\ld$ is
isomorphic to the image of $\GG^\ld\cap\GG_\sg$ in $\Sigma_{\sg}\{\pm\}$. Thus, the stratum 
map $\dot{\beta}_\sg^\ld$ is injective if and only if 
$\GG^\ld\cap \GG_{\vec{\sg}}=\GG^\ld\cap \GG_\sg$. 

In Proposition~3.1 in \cite{PFT}, it is wrongly claimed that all boundary maps
$\delta_\sg^{(m)}\ra\ccM^{(m)}$ are injective for $m\geq 3$. What is actually proved there
is the weaker statement (which is also an immediate consequence of Corollary~1.8 \cite{Ivanov}):

\begin{proposition}\label{without rotations}Let $\GG^\ld$ be a level of $\GG_{g,n}$ contained
in some abelian level $\GG(m)$, for $m\geq 3$. Then, the group $\GG^\ld$ operates without 
inversions on the curve complex $C(S_{g,n})$.
\end{proposition}
\begin{proof}Let us show first that, given two disjoint oriented separating s.c.c.'s $\vec{\gamma}_0$ 
and $\vec{\gamma}_1$ which either have distinct free homotopy classes or differ in their 
orientations, there is no $f\in\GG^\ld$ such that $f(\vec{\gm}_0)=\vec{\gm}_1$.
Let us denote by $S_0$ and $S_1$ two disjoint subsurfaces of $S_{g,n}$ such that
$\gamma_0=\dd S_0$ and $\gm_1=\dd S_1$, respectively. If there is an $f\in\GG^{\ld}$ such that
$f(\vec{\gamma}_0)=\vec{\gamma}_1$,  then the homeomorphism $f$ maps $S_0$ on $S_1$. If
$S_0$ or $S_1$ are punctured, this is not possible. If they are unpunctured, then their genus is at
least $1$ and $f$ acts non-trivially on the homology of $S_g$, with any system of coefficients. Hence
$f\notin\GG^\ld\leq\GG(m)$, for $m\geq 2$.

The only inversions, for the action of $\GG^\ld$ on $C(S_{g,n})$, may occur on non-separating 
s.c.c.'s. But then the only possibility is that an $f\in\GG^\ld$ swops two disjoint homologous 
s.c.c.'s, i.e. a cut pair $\{\gm_0,\gm_1\}$. Let $S_{g,n}\ssm\{\gm_0,\gm_1\}\cong S_1\amalg S_2$. 
If $f$ swops $\gm_0$ and $\gm_1$, but not $S_1$ and $S_2$, it is easy to see that 
$f(\vec{\gm}_0)=-\vec{\gm}_0$ in the homology group $H_1(S_{g,n},\Z/m)$. Therefore, 
$f\notin\GG(m)$, for $m\geq 3$. If instead $f$ swops $S_1$ and $S_2$, arguing as 
above, we conclude that $f\notin\GG(m)$, for $m\geq 2$. 

\end{proof}

The statement of Proposition~3.1 in \cite{PFT} holds with the following additional hypotheses:

\begin{proposition}\label{no self-intersection}For $2g-2+n>0$, let $\ccM^\ld$ be a level structure 
which dominates a smooth level structure $\ccM^\mu$ over $\ccM_{g,n}$, such that 
$\GG^\mu\leq\GG(m)$, for some $m\geq 2$. Then, all boundary maps 
$\delta_\sg^\ld\ra\ccM^\ld$ are embeddings. In particular, the level $\GG^\ld$ operates 
without inversions on $C(S_{g,n})$.
\end{proposition}
\begin{proof}By arguments similar to those in the proof of Proposition~\ref{without rotations}, 
one is reduced to show that, for any two disjoint, non-isotopic, non-separating s.c.c.'s ${\gm}_0$ 
and ${\gm}_1$ on $S_{g,n}$, there is no $f\in\GG^\ld$ such that $f({\gm}_0)={\gm}_1$.

By Theorem~\ref{monodromy coefficients}, the natural morphism $\ccM^\mu\ra\ccM_{g,n}$ ramifies 
over the divisor of nodal irreducible curves. But, in 
the proof of Proposition~2.1 \cite{B-P}, it was shown that an $f\in\GG^\mu$ with the above
properties would then yield a singularity in the stratum of the D--M boundary of $\ccM^\mu$, 
corresponding to the simplex $\{\gm_0,\gm_1\}$, against our hypotheses. 
Hence, there is no such an $f\in\GG^\mu$ and, a fortiori, in $\GG^\ld$.

\end{proof}

\begin{remark}\label{simplicial nerve}Let $\GG^\ld$ be a level of $\GG_{g,n}$ satisfying
the hypotheses of either Proposition~\ref{without rotations} or Proposition~\ref{no self-intersection}. 
Then, the simplicial set $C^\ld(S_{g,n})_\bt$ of Definition~\ref{nervedef} can be naturally realized
as a quotient in the category of simplicial sets.
Let us order the vertices of the simplicial complex $C(S_{g,n})$ compatibly with 
the action of $\GG^\ld$ and let $C(S_{g,n})_\bt$ be the simplicial set associated to the simplicial 
complex $C(S_{g,n})$ and this ordering of its vertices. The finite simplicial set 
$C^\ld(S_{g,n})_\bt$ is then naturally isomorphic to the quotient of $C(S_{g,n})_\bt$ 
by the simplicial action of the level $\GG^\ld$. 
\end{remark}

\section{The D--M boundary of abelian level structures}\label{abelian boundary structure}
In this section, we are going to describe the boundary components of abelian level
structures in terms of other simple geometric levels. Let us recall the statement of 
Lemma~3.4 \cite{mod}:

\begin{lemma}\label{comparison lemma}\begin{enumerate}
\item Let $\GG^\ld$ be a level of $\GG_{g,n}$ such that the branch locus of the covering
$\ccM^\ld\ra\ccM_{g,n}$ is contained in the boundary divisor of nodal irreducible curves. Let 
$p\co\Gamma_{g,n}\ra\Gamma_{g,n-1}$ be the epimorphism induced filling in $P_n$ on $S_{g,n}$ and 
let $\ccM^{\ld'}$ be the level structure over $\ccM_{g,n-1}$ associated to the level $p(\Gamma^\ld)$ of 
$\GG_{g,n-1}$. Then, it holds $\pi_1(\ccM^\ld)=\pi_1(\ccM^{\ld'})$. 
\item Let $\GG^{\ld_1}\leq\GG^{\ld_2}$ be two levels of $\GG_{g,n}$ whose associated level 
structures satisfy the hypothesis in the above item. If the natural morphism 
$\ccM^{\ld_1}\ra\ccM^{\ld_2}$ is {\'e}tale and, with the same notations as above, it holds
$p(\Gamma^{\lambda_1})=p(\Gamma^{\lambda_2})$, then $\GG^{\ld_1}=\GG^{\ld_2}$.
\end{enumerate}
\end{lemma}

Let $N$ be the kernel of the epimorphism $\Pi_{g-1,n+2}\tura\Pi_{g-1,2}$, induced
filling in the punctures $P_1,\ldots, P_n$. Let us then define the $\GG_{g-1,n+2}$-invariant 
normal subgroup of $\Pi_{g-1,n+2}$:
$$\Pi^{(m)_0}:=N\cdot\Pi_{g-1,n+2}^{[2],m}.$$
Let $\GG^{(m)_0}\unlhd\GG_{g-1,n+2}$ be the associated geometric level.
For $n=0$, it holds $\GG^{(m)_0}=\GG^{[2],m}$. 

\begin{theorem}\label{abelian boundary}
For $2g-2+n>0$, let $\ccM_{g,n}^{(m)}$ be an abelian level structure of order $m\geq 2$.
\begin{enumerate}
\item Let $\alpha$ be a separating circle on $S_{g,n}$, such that $S_{g,n}\ssm\alpha\cong
S_{g_1,n_1+1}\amalg S_{g_2,n_2+1}$, and let
$\beta_\alpha^{(m)}:\delta_\alpha^{(m)}\ra\ccM_{g,n}^{(m)}$ be the associated boundary map. 
Then, the morphism $\beta_\alpha^{(m)}$ is an embedding and there is a natural isomorphism
$\delta_\alpha^{(m)}\cong\ccM_{g_1,n_1+1}^{(m)}\times\ccM_{g_2,n_2+1}^{(m)}$.
\item Let $\gamma$ be a non-separating circle on $S_{g,n}$ and
$\beta_\gm^{(m)}:\delta_\gm^{(m)}\ra\ccM_{g,n}^{(m)}$ be the associated boundary map. Then,
the morphism $\beta_\gm^{(m)}$, for $g>2$, is never an embedding, and, with the above
notations, there is a natural isomorphism $\delta_\gm^{(m)}\cong\ccM_{g-1,n+2}^{(m)_0}$.
\end{enumerate}
\end{theorem}

\begin{proof}The homology $H_1(C,\Z/m)$ of a stable $n$-pointed genus 
$g$ curve of compact type $C$ has a natural structure of non-degenerate symplectic free 
$\Z/m$-module of rank $2g$ and the decomposition of $C$ in irreducible components determines 
a symplectic decomposition of this module. 

Let $\wt{\cM}_{g,n}^{(m)}$ the substack of $\ccM_{g,n}^{(m)}$ parametrizing curves of compact type.
Then, $\wt{\cM}_{g,n}^{(m)}$ is the moduli space of stable, $n$-pointed, genus $g$ curves of
compact type endowed with a symplectic isomorphism 
$H_1(C,\Z/m)\sr{\sim}{\ra}((\Z/m)^{2g},\langle\_,\_\rangle)$, where $\langle\_,\_\rangle$ is the 
standard symplectic form. It is also clear that the clutching morphism defines an embedding of
the product $\wt{\cM}_{g_1,n_1+1}^{(m)}\times\wt{\cM}_{g_2,n_2+1}^{(m)}$ inside 
$\wt{\cM}_{g,n}^{(m)}$. Since the boundary map $\delta_\alpha^{(m)}\ra\ccM_{g,n}^{(m)}$ is also 
an embedding, by the proof of 
Proposition~\ref{without rotations}, the first part of the theorem follows. 
 
As to ii.), from Proposition~6.7 in \cite{Putman}, it follows that, if $\gm$ and $\gm'$ form a cut pair 
on $S_{g,n}$, then there is an $f\in\GG(m)$ such that $f(\gm)=\gm'$. By the description of the
D--M boundary of level structures given in Section~\ref{simplicial}, we then know that such cut pair
determines a self-intersection on the boundary component of $\ccM_{g,n}^{(m)}$ parametrized
by $\gm$. 

Let us then prove the second part of ii.).
The D--M stack $\delta_\gm^{(m)}$ is normal, hence naturally isomorphic to the compactification
of the level structure defined by the \'etale Galois covering $\od_\gm^{(m)}\ra\cM_{g-1,n+2}$. 
By Section~\ref{simplicial} and Theorem~\ref{monodromy coefficients}, there is a short exact sequence
$$1\ra\Z\cdot\tau_\gamma^m\ra\GG(m)\cap \GG_{\vec{\gm}}\sr{q}{\ra}\pi_1(\od_\gm^{(m)})\ra 1.$$

Fix a homeomorphism $h\co S_{g,n}\ssm\gamma\ra S_{g-1,n+2}$ mapping the punctures of
$S_{g,n}$ in the first $n$ punctures of $S_{g-1,n+2}$. This induces a monomorphism
$h_\sharp\co\Pi_{g-1,n+2}\hookra\Pi_{g,n}$ and then a monomorphism
$$h_\sharp\co\left.\Pi_{g-1,n+2}\right/\Pi^{(m)_0}\hookra H_1(S_g,\Z/m).$$ 

Let $f\in\GG(m)\cap \GG_{\vec{\gm}}$ then $f$ acts trivially on $H_1(S_g,\Z/m)$. Therefore, $f$ acts
trivially also on the subgroup $\Pi_{g-1,n+2}/\Pi^{(m)_0}$, hence
$q(f)\in\GG^{(m)_0}$ and the level $\pi_1(\od_\gm^{(m)})$  is contained in $\GG^{(m)_0}$.
To prove that the two levels are the same, let us use Lemma~\ref{comparison lemma}. 

Denote by $p\co\GG_{g-1,n+2}\ra\GG_{g-1,n+1}$ the
epimorphism induced filling in $P_{n+2}$ and by $\GG^\ld$ the level in $\GG_{g-1,n+1}$ given 
by the subgroup $p(\pi_1(\od_\gm^{(m)}))$. We claim that $\GG^\ld=p(\GG^{(m)_0})$. 

It is easy to check that $p(\GG^{(m)_0})=\GG(m)$. In particular, it holds $\GG^\ld\leq\GG(m)$. 
For the reverse inclusion, let us observe that, by the first part of the theorem, $\delta_\gm^{(m)}$ 
contains a stratum of $\ccM_{g,n}^{(m)}$ isomorphic to $\cM_{0,3}\times\cM_{g-1,n+1}^{(m)}$. 

Therefore, the natural morphism $\delta_\gm^{(m)}\ra\ccM_{g-1,n+1}^\ld$ induces by restriction an 
\'etale morphism $\cM_{g-1,n+1}^{(m)}\ra\cM_{g-1,n+1}^\ld$. This implies that $\GG(m)\leq\GG^\ld$
and then $\GG^\ld=p(\GG^{(m)_0})$.

It remains to prove that the morphism 
$\delta_\gm^{(m)}\ra\ccM_{g-1,n+2}^{(m)_0}$ is \'etale and that the branch loci of both levels
over $\ccM_{g-1,n+2}$ are contained in the divisor $\beta_0(\ccM_{g-2,n+4})$. By the first part 
of the theorem, it is enough to prove both statements for $n=0$. Since, as noticed above, in 
this case it holds $\GG^{(m)_0}=\GG^{[2],m}$, the conclusion follows from the computation of 
the local monodromy coefficients for these levels (cf. Theorem~3.3.3 in \cite{P}).

\end{proof}

\section[The nerve of the boundary of abelian level structures]{The nerve of the D--M boundary 
of abelian level structures}\label{nerve}

There is no obvious way to extend the description of codimension $1$ strata given in
Theorem~\ref{abelian boundary} to higher codimension. At any rate, in this section, we
are going to describe the nerve of the D--M boundary of the abelian levels 
in an explicit if complicated way. 

In Section~\ref{simplicial}, we saw that, for all $m\geq 2$, the nerve of the boundary of 
$\ccM_{g,n}^{(m)}$ is the simplicial set $C^{(m)}(S_{g,n})_\bt$ associated 
to the quotient cellular complex $|C(S_{g,n})|/\GG(m)$. 

As a preliminary step, we need to describe the quotient of the curve complex $C(S_{g,n})$ 
by the action of the Torelli group $\cT_{g,n}$, which is the kernel of the natural representation
$\GG_{g,n}\ra\Sp_{2g}(\Z)$. Since the Torelli group $\cT_{g,n}$ is contained in the abelian level 
$\GG(3)$, it acts without inversions on the curve complex. So, let $C(S_{g,n})_\bt$ be the 
simplicial set associated to the curve complex and an order of 
its vertices compatible with the action of the Torelli group $\cT_{g,n}$. We then define 
$C^{(0)}(S_{g,n})_\bt$ to be the quotient simplicial set $C(S_{g,n})_\bt/\cT_{g,n}$.

Let us recall the definition of a graph of groups and of the fundamental group of a graph of groups 
(cf. Ch. I, \S 5.1 \cite{Serre}). For a graph $Y$, let us denote
by $v(Y)$ its vertex set, by $e(Y)$ its set of unoriented edges and by $\vec{e}(Y)$ its set of oriented 
edges. For an oriented edge $\vec{e}\in\vec{e}(Y)$ we denote by $\cev{e}\in\vec{e}(Y)$ 
the same edge with the opposite orientation and by $e\in e(Y)$ the underlying unoriented edge.

A graph of groups $(G,Y)$ is given by a graph $Y$ and the data of groups $G_v$ and $G_e$, for 
every vertex $v\in v(Y)$ and every unoriented edge $e\in e(Y)$, together with monomorphisms 
$s_{\vec{e}}\co G_e\hookra G_{s(\vec{e})}$ and $t_{\vec{e}}\co G_e\hookra G_{t(\vec{e})}$, for 
every oriented edge $\vec{e}\in\vec{e}(Y)$, where $s(\vec{e})$ and $t(\vec{e})$ are 
the vertices of $e$ in the source and in the target, respectively, of the given orientation. 

The fundamental group $\pi_1(G,Y)$ of the graph of groups $(G,Y)$ is the group with presentation
$\langle P|R\rangle$, where $P$ is the free product of all vertex groups $G_v$ and of the free
group on the oriented edges $\vec{e}(Y)$, while $R$ is the set of relations $\cev{e}=\vec{e}^{-1}$ 
and $\vec{e}\cdot t_{\vec{e}}(a)\cdot\cev{e}=s_{\vec{e}}(a)$, 
for all $e\in e(Y)$ and $a\in G_e$.

\begin{definition}\label{graph deco}For $m\neq 1$ a non-negative integer and $g\geq 1$,
let $H_m:=H_1(S_g,\Z/m)$ (in particular, $H_0:=H_1(S_g,\Z)$), endowed with the symplectic 
form $\langle\_,\_\rangle$ induced by intersection of cycles on $S_g$. Let
$\cP_n$ be the set of subsets of the set $\{1,\ldots,n\}$.

For $2g-2+n>0$, an {\it $n$-marked graph decomposition} $(G,Y,d)$ of $H_m$ is a graph of groups
$(G,Y)$, whose vertex and edge groups $\{G_\sg\}$ are primitive subgroups of $H_m$, endowed 
with a marking $d\co v(Y)\ra\cP_n$. We require that the following properties be satisfied.

\begin{enumerate}
\item There is an isomorphism $\phi\co\pi_1(G,Y)^{\mathrm{ab}}\otimes\Z/m\sr{\sim}{\ra}H_m$ such 
that it holds $\phi|_{G_\sg}=\mathrm{id}_{G_\sg}$, for all simplices $\sg$ of the graph $Y$, where 
the superscript "ab" is for abelianization.
\item It holds $\bigcup_{v\in v(Y)}d(v)=\{1,\ldots,n\}$ and 
$d(v)\cap d(v')=\emptyset$, for $v\neq v'\in v(Y)$.
\item Let $E$ be the subgroup of $H_m$ generated by all edge groups of $(G,Y)$. Then, for all 
vertices $v$ of $Y$, it holds $\langle x,y\rangle=0$, for all $x\in E$ and $y\in G_v$. In particular, 
$E$ is an isotropic sub-lattice of $H_m$.
\item All edge groups are cyclic. Let $\{e_i\}_{i\in I}$ be a set of unoriented edges. 
Then, if $Y\ssm\cup_{i\in I}e_i$
is connected, the associated set of edge groups $\{G_{e_i}\}_{i\in I}$ spans a primitive
subgroup of $H_m$ of rank $|I|$. If instead $\{e_i\}_{i\in I}$ is minimal
for the property that $Y\ssm\cup_{i\in I}e_i$ is disconnected, there is a set $\{u_i\}_{i\in I}$ 
of generators of these edge groups such that it holds $\sum_{i\in I}u_i=0$.
\item For $v$ a vertex of $Y$, let $\mathrm{star}(v)$ be its star in $Y$ and let $|\mathrm{star}(v)|$
be the number of edges contained in the star. Then, all vertex groups $G_v$, such that 
$|\mathrm{star}(v)|+|d(v)|\leq 2$, are non trivial and of rank $\geq 2$.
\item For $v$ a vertex of $Y$, let $E(v)$ be the subgroup of $G_v$ generated by the edge groups
it contains. Then, there is a decomposition $G_v=E(v)\oplus G_v'$, where $G_v'$ is a
symplectic, non-degenerate, possibly trivial, sub-module of $H_m$.
\end{enumerate}
The {\it rank} of an $n$-marked graph decomposition $(G,Y,d)$ is the cardinality $|e(Y)|$ 
of the set of unoriented edges of $Y$. 
We say that two $n$-marked graph decompositions $(G,Y,d)$ and $(G',Y',d')$ of $H_m$
are equivalent if there is an isomorphism $f\co Y\ra Y'$ of the underlying graphs such that, for all
$v\in v(Y)$, it holds $d'(f(v))=d(v)$ and, for every simplex $\sg$ of $Y$, it holds $G'_{f(\sg)}=G_\sg$. 
Let us denote by $[G,Y,d]$ the equivalence class of the $n$-marked graph decomposition 
$(G,Y,d)$. If $n=0$, we simply say that $(G,Y)$ is a graph decomposition of $H_m$ and denote
by $[G,Y]$ the corresponding equivalence class.
\end{definition}

A $k$-simplex $\sg\in C(S_{g,n})$ determines a partition $\amalg_{v\in V} S_v:=S_{g,n}\ssm\sg$ of 
the surface $S_{g,n}$. We can associate to $\sg$ the graph $(G,Y)$ of abelian subgroups of 
$H_1(S_g,\Z/m)$ whose vertex groups consist of the images of the natural homomorphisms
$H_1(S_v,\Z/m)\ra H_1(S_g,\Z/m)$, for $v\in V$, and whose edge groups are the cyclic subgroups of
$H_1(S_g,\Z/m)$ determined by the s.c.c.'s in $\sg$. The marking $d\co v(Y)=V\ra\cP_n$ is
then defined assigning to a vertex $v$ of the graph $Y$ the indices of the punctures on the corresponding
subsurface $S_v$ of $S_{g,n}$. 

In this way, it is associated to every simplex $\sg\in C(S_{g,n})$ 
an $n$-marked graph decomposition $(G,Y,d)$ of rank $k+1$.

Conversely, it is easy to check that, given an $n$-marked graph decomposition 
$(G,Y,d)$ of $H_m$ of rank $k+1$, there is a $k$-simplex $\sg\in C(S_{g,n})$ which has $(G,Y,d)$ 
for associated $n$-marked graph decomposition.

The partition of $S_{g,n}$ determined by a simplex $\sg\in C(S_{g,n})$ is refined by the partition 
determined by a simplex $\sg'$ such that $\sg\subset\sg'$. 

\begin{definition}\label{order} 
Let us define a partial order on the set of all $n$-marked graph decompositions of $H_m$, letting 
$(G',Y',d')\leq (G,Y,d)$ if the following conditions are satisfied:
\begin{enumerate}
\item the set of edge groups of $(G',Y',d')$ is a subset of the set of edge groups of $(G,Y,d)$;
\item for each vertex $v'$ of $Y'$, there is a connected subgraph of groups $(K,Z)$ of $(G,Y)$ 
such that $d'(v')=\bigcup_{v\in v(Z)}d(v)$ and there is an isomorphism
$\psi\co\pi_1(K,Z)^{\mathrm{ab}}\otimes\Z/m\sr{\sim}{\ra}G'_{v'}$ with the property that
$\psi|_{K_\sg}=\mathrm{id}_{K_\sg}$ for all simplices $\sg$ of the subgraph $Z$.
\end{enumerate}

The partial order so defined on the set of $n$-marked graph decompositions of $H_m$ 
is compatible with the equivalence relation defined above and thus induces a partial order on 
the set of equivalence classes of $n$-marked graph decompositions of $H_m$. 
\end{definition}

\begin{definition}\label{complex graph deco}
For $m\neq 1$ a non-negative integer, let $\cG^n(H_m)_\bt$ be the simplicial 
set whose semi-simplicial set of non-degenerate simplices is defined as follows.
\begin{enumerate}
\item The set of $k$-simplices is the set of equivalence classes of $n$-marked graph 
decompositions of $H_m$ of rank $k+1$. 
\item The $(k-1)$-faces of a $k$-simplex $[G,Y,d]$  are the equivalence 
classes of $n$-marked graph decompositions $(G',Y',d')\leq (G,Y,d)$ of rank $k$.  A face map 
$\dd_e\co [G,Y,d]\ra [G',Y',d']$ is assigned to the unoriented edge $e\in e(Y)$, if the set of edge 
groups of $[G',Y',d']$ coincides with the set of edge groups associated to the edges of $Y\ssm e$. 
\end{enumerate}
\end{definition}

\begin{theorem}\label{abelian nerve}For $2g-2+n>0$ and $g\geq 1$, there is a natural isomorphism 
of simplicial sets $\Psi_\bt\co C^{(0)}(S_{g,n})_\bt\ra\cG^n(H_0)_\bt$.
\end{theorem}

\begin{proof}By definition, $C^{(0)}(S_{g,n})_k=C(S_{g,n})_k/\cT_{g,n}$.
From the above remarks, it follows that, for $k\geq 0$, there 
is a natural surjective map $\widetilde{\Psi}_k\co C(S_g)_k\ra\cG^n(H)_k$. It is easy to 
check the compatibility of $\widetilde{\Psi}_\bt$ with the face maps so that we get a map of 
simplicial sets. In order to prove that $\widetilde{\Psi}_\bt$ is an isomorphism,
we have to show that its fibers coincide with the orbits of the action of $\cT_{g,n}$ on 
$C(S_{g,n})_\bt$. 

Since the Torelli group $\cT_{g,n}$ acts trivially on the set of $n$-marked graph decompositions 
of $H_0$, for all $k\geq 0$, the map $\widetilde{\Psi}_k$ factors through the map:
$$\Psi_k\co\left.C(S_{g,n})_k\right/\cT_{g,n}\ra\cG^n(H_0)_k.$$

Let us then show that two simplices $\sg$ and $\sg'$ which define equivalent 
$n$-marked graph decompositions of $H$ of rank $k+1$ are in the same $\cT_{g,n}$-orbit. 
Let us choose a representative $(G,Y,d)$ of the equivalence class and let us identify the two 
isomorphic $n$-marked graph decompositions of $H$ with it.

Let $\sg=\{\gm_0,\ldots,\gm_k\}$ and $\sg'=\{\gm'_0,\ldots,\gm'_k\}$.
With suitable orientations, the s.c.c.'s in the simplices $\sg$ and 
$\sg'$ determine the same set of not necessarily linearly independent homology classes 
$\{a_0,\ldots,a_k\}$ of $H_0$. Let us then extend this set to a generating set $\{a_0,\ldots,a_s\}$  
of $H_0$ with the following properties.

\begin{itemize}
\item Each vertex group of $(G,Y,d)$ is naturally isomorphic to the homology of a subsurface of 
$S_g$, whose topological type is a punctured surface $S_{g',n'}$, with $2g'-2+n'>0$. 
Let us then define a standard set of generators $\{a_1,b_1,\ldots,a_{g'},b_{g'},w_1,\ldots,w_{n'}\}$ 
for the group $H_1(S_{g',n'},\Z)$, where, for $n'=1$, we omit $w_1$, to be a set of generators 
which can be lifted to a standard set of generators for the fundamental group (cf. \S \ref{levels}). 

The intersection $G_y\cap\{a_0,\ldots,a_s\}$ is then
a standard set of generators for the group $G_y$, for every vertex $y$ of $Y$.

\item There is a natural epimorphism $r\co\pi_1(G,Y)^{\mathrm{ab}}\tura H_1(Y,\Z)$.
The set $\{a_{i_1},\ldots,a_{i_t}\}$ of elements in $\{a_0,\ldots,a_s\}$ which do not belong 
to any vertex group of $(G,Y,d)$ is then projected by $r$ to a basis of $H_1(Y,\Z)$ such that 
$r(a_{i_j})$, for $j=1,\ldots,t$, is the cycle associated to a circuit in the graph $Y$ which does
not pass twice for the same vertex.
Moreover, it holds $\langle a_{i_j},a_l\rangle=0$, for all $j=1,\ldots,t$ and $l>k$.
\end{itemize}

By the second property, the set of homology classes $\{a_{i_1},\ldots,a_{i_t}\}$ can be lifted 
to a set of disjoint s.c.c.'s $\{\gm_{i_1},\ldots,\gm_{i_t}\}$ (respectively, 
$\{\gm'_{i_1},\ldots,\gm'_{i_t}\}$) such that, for $j=1,\ldots,t$, each $\gm_{i_j}$ (respectively, each 
$\gm'_{i_j}$) intersects a curve in $\sg$ (respectively, in $\sg'$) at most once. Let us extend 
these sets of disjoint s.c.c.'s to sets of s.c.c.'s
$\{\gm_0,\ldots,\gm_s\}$ and $\{\gm'_0,\ldots,\gm'_s\}$ which moreover satisfy:
\begin{itemize}
\item 
for the associated integral homology classes, it holds $[\gm_i]=[\gm_i']=a_i$, for $i=0,\ldots,s$;
\item it holds $|\gm_i\cap_g\gm_j|=|\gm_i'\cap_g\gm_j'|=|\langle a_i,a_j\rangle|$,  
for $i,j=0,\ldots,s$, where "$\cap_g$" denotes geometric intersection;
\end{itemize}

For a given vertex $y$ of $Y$, let $S_y$ and $S_y'$ be, respectively, the closures in $S_{g,n}$ of
the connected components of $S_{g,n}\ssm\sg$ and $S_{g,n}\ssm\sg'$ whose homology 
groups map onto the vertex group $G_y$. There is then a homeomorphism 
$\phi_y\co S_y\ra S_y'$ such that it holds $\phi_y(\gm_i\cap S_y)=\gm_i'\cap S_y'$, for $i=0,\ldots,s$.

Glueing along $\sg$ and $\sg'$ these homeomorphisms, for all connected components of 
$S_{g,n}\ssm\sg$ and $S_{g,n}\ssm\sg'$, we get a homeomorphism 
$\phi\co S_{g,n}\ra S_{g,n}$ such that $\phi(\gm_i)=\gm_i'$, for $i=0,\ldots,s$. This 
homeomorphism then acts trivially on $H_0$ and is such that $\phi(\sg)=\sg'$.

\end{proof}

The situation is a bit more complicated when we want to describe the finite simplicial sets 
$C^{(m)}(S_{g,n})_\bt$, for $m\geq 2$. In this case, the set of equivalence classes of 
$n$-marked graph decompositions of $H_m$ of rank $k+1$, in general, is not in bijective 
correspondence with the set of non-degenerate $k$-simplices of $C^{(m)}(S_{g,n})_\bt$.

The set $C^{(m)}(S_{g,n})_k$, for $k\geq 0$, is the quotient of the set 
$C^{(0)}(S_{g,n})_k\cong\cG^n(H)_k$, described in Theorem~\ref{abelian nerve}, by the action 
of the congruence subgroup $\Sp_{2g}(\Z,m)\cong\GG(m)/\cT_{g,n}$ of the symplectic 
group $\Sp_{2g}(\Z)$ and there is a natural surjective map $\cG^n(H)_\bt\tura\cG(H_m)_\bt$, 
which is invariant under the action of $\Sp_{2g}(\Z,m)$. Therefore, there is at least a natural 
surjective map of simplicial sets:
$$C^{(m)}(S_{g,n})_\bt\tura\cG^n(H_m)_\bt.$$

Let then $(G,Y,d)$ and $(G',Y',d')$ be two $n$-marked graph decompositions of $H_0$ which 
induce, reducing modulo $m$, equivalent $n$-marked graph decompositions of $H_m$, i.e. 
such that there is an isomorphism of graphs $f\co Y\ra Y'$ with the property that $d(v)=d'(f(v))$, 
for all $v\in v(Y)$, and it holds $G_\sg/m=G'_{f(\sg)}/m$, for all simplices $\sg$ of $Y$, where, 
for a subgroup $K$ of $H_0$, we denote by $K/m$ its image via the natural epimorphism 
$H_0\tura H_m$. 

The corresponding equivalence classes $[G,Y,d]$ and $[G',Y',d']$ of $n$-marked graph 
decompositions of $H_0$ are then in the same $\Sp_{2g}(\Z,m)$-orbit if and only if there is an 
$F\in\Sp_{2g}(\Z,m)$ 
such that $F(G_\sg)=G'_{f(\sg)}$. This happens if and only if there are standard symplectic basis 
$\{a_1,\ldots,a_{2g}\}$ and $\{a'_1,\ldots,a'_{2g}\}$ of $H_0$ with the following properties:
\begin{enumerate}
\item[a)] it holds $a_i\equiv a_i' \!\mod m$, for $i=1,\ldots,2g$;
\item[b)] the set $\{a_1,\ldots,a_{2g}\}$ (resp. $\{a'_1,\ldots,a'_{2g}\}$) contains generators for
the edge groups in a maximal set of linearly independent edge groups of $(G,Y)$ (resp. of $(G',Y')$);
\item[c)] for every vertex $v\in Y$, there is a set of indices $\{i_1,\ldots i_k\}$ 
such that it holds, simultaneously, $G_v\cap\{a_1,\ldots,a_{2g}\}=\{a_{i_1},\ldots,a_{i_k}\}$ and
$G'_{f(v)}\cap\{a'_1,\ldots,a'_{2g}\}=\{a'_{i_1},\ldots,a'_{i_k}\}$, and the projections of these 
intersections to the quotients $G_v/E(v)$ and $G'_{f(v)}/E(f(v))$, see
 vi.) Definition~\ref{graph deco}, generate these groups.
\end{enumerate}
By property a), the assignment $a_i\mapsto a'_i$, for $i=1,\ldots,2g$, defines an
$F\in\Sp_{2g}(\Z,m)$. Properties b) and c) make sure that $F(G_\sg)=G'_{f(\sg)}$, for $\sg$ a vertex
or an edge of $Y$.

Let, as above, $(G,Y,d)$ and $(G',Y',d')$ be two $n$-marked graph decompositions of $H_0$ which 
induce, reducing modulo $m$, equivalent $n$-marked graph decompositions of $H_m$, related by 
the isomorphism of graphs $f\co Y\ra Y'$.
Reduction modulo $m$ then defines, for every simplex $\sg$ of $Y$, epimorphisms
$\mu_\sg\co G_\sg\ra G_\sg/m$ and $\mu'_{f(\sg)}\co G'_{f(\sg)}\ra G'_{f(\sg)}/m$ of $\Z$-modules, 
where $G_\sg/m=G'_{f(\sg)}/m$ is a free primitive $\Z/m$-submodule of $H_m$.

Let us observe that, given a vertex $v$ of $Y$, any basis for a maximal symplectic non-degenerate 
free $\Z/m$-submodule of $G_v/m=G'_{f(v)}/m$ can be lifted at the same time, via the epimorphisms 
$\mu_v$ and $\mu'_{f(v)}$, to bases for maximal symplectic non-degenerate free $\Z$-submodules 
of $G_v$ and $G'_{f(v)}$, respectively.

On the other hand, given an edge $e$ of $Y$, for $m>3$, it is not necessarily the case that there is 
a generator for the edge group $G_e/m=G'_{f(e)}/m$ which can be lifted both to a generator of the 
edge group $G_e$, via the epimorphism $\mu_e$, and to a generator of the edge group $G'_{f(e)}$, 
via the epimorphism $\mu'_{f(e)}$.

However, if the last-mentioned condition is fulfilled for all edges, then there are 
bases $\{a_1,\ldots,a_{2g}\}$ and $\{a'_1,\ldots,a'_{2g}\}$ of $H_0$ which satisfy properties a), 
b) and c) with respect to the two given $n$-marked graph decompositions of $H_0$. 
This is the motivation for the following definitions:

\begin{definition}\label{marked graph}A framed, $n$-marked graph decomposition 
$(G,Y,d,\{\mu_e\})$ of $H_m$ is the data of an $n$-marked graph decomposition $(G,Y,d)$ of 
$H_m$ and, for every unoriented edge $e$ of $Y$, an epimorphism $\mu_e\co\Z\ra G_e$ 
(a {\it frame} for the edge group $G_e$).

Two framed, $n$-marked graph decompositions $(G,Y,d,\{\mu_e\})$ and 
$(G',Y',d',\{\mu'_e\})$ of $H_m$ are equivalent if there is an isomorphism $f\co Y\ra Y'$ of the 
underlying graphs such that it holds $d'(f(v))=d(v)$, for all $v\in v(Y)$, it holds $G_\sg=G'_{f(\sg)}$, 
for all simplices $\sg$ of $Y$, and it holds $\mu_{f(e)}(1)=\pm\mu_e(1)$, for every $e\in e(Y)$.
We then denote by $[G,Y,d,\{\mu_e\}]$ the equivalence class of $(G,Y,d,\{\mu_e\})$.
\end{definition}

It is easy to check that, given a framed, $n$-marked graph decomposition $(G,Y,d,\{\mu_e\})$ 
of $H_m$, there is an $n$-marked graph decomposition $(\tilde{G},Y,d)$ of $H_0$ such that its 
reduction modulo $m$ is the $n$-marked graph decomposition $(G,Y,d)$ and the corresponding 
natural epimorphisms of edge groups $\tilde{G}_e\tura G_e$ define frames equivalent to the 
$\mu_e$, for all $e\in e(Y)$. 

\begin{definition}\label{complex framed graph deco}
Let $\cG^n_\ast(H_m)_\bt$ be the simplicial set whose semi-simplicial set of non-degenerate 
simplices is defined as follows.
\begin{enumerate}
\item The set of $k$-simplices is the set of equivalence classes of framed, $n$-marked graph 
decompositions of $H_m$ of rank $k+1$. 

\item The $(k-1)$-faces of a $k$-simplex $[G,Y,d,\{\mu_e\}]$  are the equivalence classes of 
$n$-marked graph decompositions $(G',Y',d')\leq (G,Y,d)$ of rank $k$, endowed with the frames 
$\{\mu_{e}\}_{e\in e(Y')}$. Face maps are then defined in the same way as in 
Definition~\ref{complex graph deco}. 
\end{enumerate}
\end{definition}

From the above discussion, it follows: 

\begin{theorem}\label{abelian nerve m}For all $m\geq 2$ and $g\geq 1$, there is a natural 
isomorphism of simplicial sets $\Psi_\bt^{(m)}\co C^{(m)}(S_{g,n})_\bt\ra\cG^n_\ast(H_m)_\bt$.
\end{theorem}

\begin{remark}\label{m=2,3}Let $(G,Y,d)$ be an $n$-marked graph decomposition of $H_m$.
For $m=2,3$, there is clearly, for each unoriented edge $e$ of $Y$, only one
equivalence class of frames $\Z\tura G_e$. Therefore, for $m=2,3$, it holds 
$\cG^n_\ast(H_m)_\bt\equiv\cG^n(H_m)_\bt$.
\end{remark}

The description given in Theorem~\ref{abelian nerve} and Theorem~\ref{abelian nerve m} of the simplicial sets $C^{(m)}(S_{g,n})_\bt$, for $g\geq 1$ and $m\neq 1$, can be made more precise, 
for $g\geq 2$ and $n=0$, restricting either to the complex of non-separating curves $C_0(S_g)$ or 
to the complex of separating curves $C_{sep}(S_g)$. The first is the sub-complex of $C(S_g)$ 
which consists of simplices $\sg$ such that $S_g\ssm\sg$ is connected. The second is the full 
sub-complex of $C(S_g)$ generated by $0$-simplices $\gm$ such that $S_g\ssm\gm$ is 
disconnected.  

In Proposition~6.7 of \cite{Putman}, the quotient $C_0^{(m)}(S_g)_\bt:=C_0(S_g)_\bt/\GG(m)$ is 
described, for $m\geq 2$, as the complex of lax isotropic bases (cf. Definition~6.6 \cite{Putman}). 
In our terminology, the simplices of $\cG_\ast(H_m)$ contained in $C_0^{(m)}(S_g)_\bt$ are 
determined by their framed edge groups. In particular, $C_0^{(m)}(S_g)_\bt$ 
is the simplicial set associated to a genuine simplicial complex, equivalently, each simplex is 
determined by its vertices. 

In \cite{Kallen}, the quotient $C_{sep}^{(0)}(S_g)_\bt:=C_{sep}(S_g)_\bt/\cT_g$ is
described by a complex of tree decompositions of $H_0$. The simplicial set $\cG^n(H_0)_\bt$ 
is a natural generalization of this notion, except that the complex of tree decompositions 
is directly realized as the simplicial complex associated to the poset of tree decompositions with 
the partial order $\leq$ defined above.

In \cite{Putman} and \cite{Kallen}, the above descriptions of the complexes 
$C_0^{(m)}(S_g)$ and $C_{sep}^{(0)}(S_g)$ are introduced in order to prove that, in analogy with
the curve complexes $C_0(S_g)$ and $C_{sep}(S_g)$, they are highly connected.
More precisely, they prove that $C_0^{(m)}(S_g)$ is $(g-2)$-connected, for 
$m\geq 2$, and that $C_{sep}^{(0)}(S_g)$ is $(g-3)$-connected, respectively. 

Thus, it is a natural question whether similar results hold for the simplicial sets 
$C^{(m)}(S_{g,n})_\bt$, for $m\neq 1$. From Proposition~3.3 \cite{B-P} and Proposition~6.4
\cite{Putman}, it follows:

\begin{proposition}\label{BP}For $g\geq 2$ and $m\neq 1$, the simplicial set 
$C^{(m)}(S_{g,n})_\bt$ is simply connected and $(g-2)$-connected. 
\end{proposition}
 
\begin{proof}By Proposition~3.3 in \cite{B-P}, the abelian level $\GG(m)$ of $\GG_{g,n}$, 
for $g\geq 2$ and $m\neq 1$, is generated by Dehn twists along separating maps, $m$-th powers 
of Dehn twists and cut pair maps. All these elements stabilize some simplex of the curve
complex $C(S_{g,n})$ and then some point of its geometric realization $|C(S_{g,n})|$.
From Theorem~3 in \cite{Armstrong}, it follows that the quotient space $|C(S_{g,n})|/\GG(m)$ is 
simply connected for $g\geq 2$ and $m\neq 1$.

As A. Putman pointed out to me, from Proposition~6.4
\cite{Putman} and the fact that, for $g\geq 2$, the $(g-1)$-skeleton of the geometric realization  
$|C^{(m)}(S_{g,n})_\bt|$ can be deformed, inside this complex, to the sub-complex 
$|C_0^{(m)}(S_g)|$, it follows that the simplicial set $C^{(m)}(S_{g,n})_\bt$ is also 
$(g-2)$-connected. 

\end{proof}

\section{The congruence subgroup problem}\label{congruence}
In this section, we are going to connect all the topics of the previous sections with the congruence
subgroup problem for the modular Teichm\"uller group $\GG_{g,n}$. This is the problem of
determining whether the tower of geometric levels $\{\GG^K\}_{K\unlhd\Pi_{g,n}}$ is
cofinal inside the tower of all finite index subgroups of $\GG_{g,n}$. This is known to be true
in genus $\leq 2$ (cf. \cite{Asada} and \cite{Hyp2}) but it is still an open problem in 
general. The solution proposed in \cite{PFT} is flawed by a mistake in the proof of 
Theorem~5.4. However, the strategy elaborated there is still viable and will be explained shortly.

In \cite{PFT}, to every profinite completion $\GG'_{g,n}$ of the Teichm\"uller group $\GG_{g,n}$ 
satisfying some mild technical conditions (cf. Definition~5.2 \cite{PFT}), is associated the
$\GG'_{g,n}$-completion $C'(S_{g,n})_\bt$ of the curve complex $C(S_{g,n})$. 

Let $\{\GG^\ld\}_{\ld\in\Ld'}$ be the set of levels of $\GG_{g,n}$ obtained as the inverse images  
of open normal subgroups of $\GG'_{g,n}$, via the natural homomorphism $\GG_{g,n}\ra\GG'_{g,n}$, 
and contained in an abelian level of order at least $3$. Then, the $\GG'_{g,n}$-completion
of the curve complex $C(S_{g,n})$ is defined to be the inverse limit of finite simplicial sets 
(cf. Remark~\ref{simplicial nerve}):
$$C'(S_{g,n})_\bt:=\lim\limits_{\sr{\textstyle\st\longleftarrow}{\sst\ld\in\Ld'}}\,C^\ld(S_{g,n})_\bt.$$
Thus, it is an object in the category of simplicial profinite sets.

In \cite{Quick} and \cite{Quick2}, Gereon Quick developed a homotopy theory for simplicial profinite 
sets, which turns out to be particularly useful here. For a simplicial profinite set $X_\bt$, let us then 
denote by $\hp_1(X_\bt)$ its fundamental group as defined by Quick.

Quick's profinite homotopy theory would not be so useful for us, there were not a way to relate
homotopy groups of the $\GG'_{g,n}$-completion $C'(S_{g,n})_\bt$ with those of its finite
quotients $C^\ld(S_{g,n})_\bt$. However, this issue can be solved positively thanks to the
results of \cite{Quick2}:

\begin{proposition}\label{quick}Let $\{X_\bt^\ld\}_{\ld\in\Ld}$ be a cofiltering inverse system
of simplicial finite sets and let $X_\bt:=\ilim_{\ld\in\Ld}\,X_\bt^\ld$. For all $q\geq 0$,
there is then a natural isomorphism:
$$\hp_q(X_\bt)\cong\lim\limits_{\sr{\textstyle\st\longleftarrow}{\sst\ld\in\Ld}}\hp_q(X_\bt^\ld).$$
In particular, the fundamental group of $X_\bt$ is the inverse limit 
of the profinite completions of the fundamental groups of the geometric realizations
$\{|X_\bt^\ld|\}_{\ld\in\Ld}$.
\end{proposition}

\begin{proof}In \S 3.5 of \cite{Quick2}, an explicit fibrant replacement functor is defined for the 
model structure of the category of simplicial profinite sets $\hat{\cS}$, constructed in \S 2 \cite{Quick}. 
For the given simplicial profinite set $X_\bt$, it is described as follows.

Let $\hat{G}X^\ld_\bt$, for $\ld\in\Ld$, be the free simplicial profinite loop groupoid
associated to the simplicial finite set $X^\ld_\bt$. By Theorem~3.11 \cite{Quick2}, the profinite 
classifying space $\bar{W}\hat{G}X^\ld_\bt$, for $\ld\in\Ld$, is fibrant in $\hat{\cS}$ 
and is equipped with a natural map $\hat{\eta}_\ld\co X^\ld_\bt\ra\bar{W}\hat{G}X^\ld_\bt$, which 
is a trivial cofibration in $\hat{\cS}$ and then makes of $\bar{W}\hat{G}X^\ld_\bt$, for $\ld\in\Ld$, 
an explicit functorial fibrant replacement of the simplicial finite set $X_\bt$.
Again by Theorem~3.11 \cite{Quick2}, the inverse limit
$$\bar{W}\hat{G}X_\bt:=
\lim\limits_{\sr{\textstyle\st\longleftarrow}{\sst\ld\in\Ld}}\,\bar{W}\hat{G}X^\ld_\bt$$
is equipped with a canonical map $\hat{\eta}\co X_\bt\ra\bar{W}\hat{G}X_\bt$ 
which is a trivial cofibration in $\hat{\cS}$ and makes of $\bar{W}\hat{G}X_\bt$ an explicit 
functorial fibrant replacement of the simplicial profinite set $X_\bt$. 
According to Lemma~2.18 \cite{Quick}, for all $q\geq 0$,
there is then a series of natural isomorphisms:
$$\hp_q(X_\bt)\cong\hp_q(\bar{W}\hat{G}X_\bt)\cong
\lim\limits_{\sr{\textstyle\st\longleftarrow}{\sst\ld\in\Ld}}\hp_q(\bar{W}\hat{G}X^\ld_\bt)
\cong\lim\limits_{\sr{\textstyle\st\longleftarrow}{\sst\ld\in\Ld}}\hp_q(X^\ld_\bt).$$

By Proposition~2.1 \cite{Quick}, the fundamental group  $\hat{\pi}_1(X^\ld_\bt)$ is just the 
profinite completion of the topological fundamental group of the geometric realization of the simplicial 
finite set $X^\ld_\bt$ and this implies the last claim in the proposition.

\end{proof} 

In particular, for the $\GG'_{g,n}$-completion $C'(S_{g,n})_\bt$ of the curve complex $C(S_{g,n})$,
there is a natural isomorphism, for all $q\geq 0$:
$$\hp_q(C'(S_{g,n})_\bt)\cong\lim\limits_{\sr{\textstyle\st\longleftarrow}{\sst\ld\in\Ld'}}
\hp_q(C^\ld(S_{g,n})_\bt).$$

Let $\hGG_{g,n}$ be the profinite completion of the Teichm\"uller group and let us define the 
profinite curve complex $\hC(S_{g,n})_\bt$ to be the $\hGG_{g,n}$-completion of the complex of 
curves. The following result, with a different formalism, was already proved in \cite{PFT}:

\begin{theorem}\label{simply connected}For $2g-2+n>0$ and $3g-3+n>2$, the simplicial profinite set 
$\hC(S_{g,n})_\bt$ is simply connected.
\end{theorem}

\begin{proof}Let $\{\GG^\ld\}_{\ld\in\Ld}$ be the set of all levels $\GG^\ld$ of $\GG_{g,n}$
which contain an abelian level of order at least $3$.
By definition, the simplicial profinite set $\hC(S_{g,n})_\bt$ is the inverse limit, over $\Ld$,
of the simplicial finite sets $C^\ld(S_{g,n})_\bt$. By what observed above, there is 
a natural isomorphism:
$$\hp_1(\hC(S_{g,n})_\bt)\cong\lim\limits_{\sr{\textstyle\st\longleftarrow}{\sst\ld\in\Ld}}
\hat{\pi}_1(C^\ld(S_{g,n})_\bt).$$ 

For $3g-3+n>2$, the simplicial set $C^\ld(S_{g,n})_\bt$ is the quotient of a simply connected 
simplicial set by the simplicial action of $\GG^\ld$. By Theorem~3 in Armstrong \cite{Armstrong},
there is then a natural epimorphism $\GG^\ld\tura\pi_1(C^\ld(S_{g,n})_\bt)$. 
Therefore, taking profinite completions, we get a natural epimorphism 
$\hGG^\ld\tura\hat{\pi}_1(C^\ld(S_{g,n})_\bt)$ and, passing to inverse limits, an epimorphism:
$$\lim\limits_{\sr{\textstyle\st\longleftarrow}{\sst\ld\in\Ld}}\hGG^\ld\tura\hp_1(\hC(S_{g,n})_\bt).$$
But now $\ilim_{\ld\in\Ld}\,\hGG^\ld=\cap_{\ld\in\Ld}\,\hGG^\ld=\{1\}$ and the theorem follows.

\end{proof}

Let us now denote by $\cGG_{g,n}$ the profinite completion of the Teichm\"uller group $\GG_{g,n}$
induced by the tower of geometric levels $\{\GG^K\}_{K\unlhd\Pi_{g,n}}$ and by
$\check{C}(S_{g,n})_\bt$ the $\cGG_{g,n}$-completion of the complex of curves. 
They are called, respectively, the {\it procongruence Teichm\"uller group} and the {\it procongruence 
curve complex}. An immediate corollary of Theorem~\ref{simply connected} is then:

\begin{corollary}\label{cong}The congruence subgroup property for the Teichm\"uller group 
$\GG_{g,n}$ holds for all pairs $(g,n)$ such that $g\leq k$ and $2g-2+n>0$, if and only if
it holds $\hat{\pi}_1(\check{C}(S_{g})_\bt)=\{1\}$ for all $2\leq g\leq k$.
\end{corollary}

\begin{proof}One implication is immediate from Theorem~\ref{simply connected}. To prove the 
other, let us proceed by induction on $k$. The congruence subgroup property holds in genus 
$\leq 2$. So, the statement of the corollary holds at least for $k=1,2$.

For $k\geq 3$, in order to complete the induction, we have to show that the assumptions, that the 
congruence subgroup property holds for all pairs $(g',n)$ such that $g'<g$ and that 
$\hp_1(\check{C}(S_{g})_\bt)=\{1\}$, imply the congruence subgroup property for all pairs $(g,n)$.

Let $\sg$ be a simplex of the curve complex $C(S_g)$ and let us denote with the same letter its
image in the simplicial profinite sets $\hC(S_{g})_\bt$ and $\check{C}(S_{g})_\bt$. Each
$\hGG_{g}$-orbit of simplices in $\hC(S_{g})_\bt$ and $\check{C}(S_{g})_\bt$ contains a
"discrete" simplex $\sg$.

By Proposition~6.5 and Proposition~6.6 in \cite{PFT}, the assumption that the congruence subgroup 
property holds for all pairs $(g',n)$ such that $g'<g$, implies that the stabilizer $\cGG_\sg$ 
for the action of $\cGG_{g}$ on $\check{C}(S_{g})_\bt$ is the profinite completion of the discrete 
stabilizer $\GG_\sg$. Therefore, the natural epimorphism $\hGG_\sg\tura\cGG_\sg$ is actually an 
isomorphism. This is then true for any simplex $\sg\in\hC(S_{g})_\bt$ and its image in 
$\check{C}(S_{g})_\bt$. 

Thus, if $\Psi_{g}\co\hGG_{g}\ra\cGG_{g}$ denotes the natural epimorphism, the profinite 
group $\ker\,\Psi_{g}$ acts freely on $\hC(S_{g})_\bt$
with quotient naturally isomorphic to $\check{C}(S_{g})_\bt$. 
From Corollary~2.3 in \cite{Quick}, it follows that 
$\ker\,\Psi_{g}\cong\hat{\pi}_1(\check{C}(S_{g})_\bt)=\{1\}$, i.e. the congruence subgroup property
holds for the Teichm\"uller group $\GG_{g}$. 

From Theorem~2 in \cite{Asada}, it then follows that the congruence subgroup property holds for all 
pairs $(g,n)$, completing the induction step.

\end{proof}

\begin{remark}\label{kent}With the above notations, Kent (cf. Theorem~12 \cite{Kent}) has 
recently proved that it holds $\hp_1(\check{C}(S_{g,n})_\bt)=\{1\}$ for $g\geq 2$ and $n\geq 2$. 
However, the induction argument of Corollary~\ref{cong} applies only for $n\leq 1$. 
\end{remark}

The main difficulty in trying to prove the simple connectivity of the procongruence curve 
complexes $\check{C}(S_{g,n})_\bt$, for $g>2$, is that, for a given geometric level $\GG^K$ of
$\GG_{g,n}$, there is no combinatorial description of the finite simplicial set $C(S_{g,n})/\GG^K$. 
So, it is natural to try to remedy to this situation by replacing the tower of geometric levels 
$\{\GG^K\}_{K\unlhd\Pi_{g,n}}$ of the Techm\"uller group $\GG_{g,n}$ with the equivalent tower 
$\{\GG^{K,(m)}\}_{K\unlhd\Pi_{g,n}}$ of Looijenga levels, for some $m\geq 2$. This was actually 
one of the main motivations to introduce these levels in the first place.

In Section~\ref{symmetry}, we saw that, for $m\geq 2$, the natural map 
$\ol{T}_{G_K}\ra\ccM(S_K)^{(m)}$ factors through a morphism 
$\iota_{K,(m)}\co\ccM^{K,(m)}\ra\ccM(S_K)^{(m)}$ which, according to Theorem~\ref{normal}, is an 
embedding whenever $K$ satisfies the hypotheses of ii.) Remarks~\ref{representability}. 

The morphism $\iota_{K,(m)}$ then induces a map from the set of simplices of the simplicial set 
$C^{K,(m)}(S_{g,n})_\bt$, which describes the nerve of the D--M boundary of $\ccM^{K,(m)}$, to 
the set of simplices of the simplicial set $C^{(m)}(S_K)_\bt$, which describes the nerve of the 
D--M boundary of $\ccM(S_K)^{(m)}$. So, for every $k\geq 0$, to the morphism $\iota_{K,(m)}$, 
is associated a map:
$$\mathfrak{b}_k\co C^{K,(m)}(S_{g,n})_k\ra\coprod_{h\geq 0}C^{(m)}(S_K)_h.$$

Taking into account (cf. Theorem~\ref{abelian nerve m}) the natural isomorphism of 
$C^{(m)}(S_K)_\bt$ with the semi-simplicial set 
$\cG^{n_K}_\ast(H_1(\ol{S}_K,\Z/m))_\bt$ of framed, $n_K$-marked graph decompositions, where
$n_K$ is the number of punctures on $S_K$, 
the map $\mathfrak{b}_k$ is described as follows.

For a given $k$-simplex $\sg\in C^{K,(m)}(S_{g,n})_\bt$, let $\td{\sg}\in C(S_{g,n})$ be a
lift to the curve complex. Then, to the simplex $\sg$, we associate the simplex $\mathfrak{b}_k(\sg)$ 
of $\cG^{n_K}_\ast(H_1(\ol{S}_K,\Z/m))_\bt$ defined by the equivalence class of the framed,
$n_K$-marked graph decomposition of the homology group $H_1(\ol{S}_K,\Z/m)$ determined by
the partition $S_K\ssm p_K^{-1}(\td{\sg})$ of the Riemann surface $S_K$.

It is not clear whether the map $\mathfrak{b}_k$, even in case the morphism $\iota_{K,(m)}$ is 
an embedding, is injective and if its image can be described in a simple way. 
In fact, it is injective only if the image of $\ccM^{K,(m)}$ intersects each $G_K$-invariant 
stratum of the abelian level $\ccM(S_K)^{(m)}$ in a single irreducible component and this is
not obvious.  

The equivalence class $[D,Y,d,\{\mu_e\}]$, of a framed, $n_K$-marked graph decomposition of 
$H_1(\ol{S}_K,\Z/m)$ in the image of the map $\mathfrak{b}_k$, for some $k\geq 0$, admits a 
natural action of the group $G_K$ (in the sense of Definition~\ref{approximation}). 

Conversely, the equivalence class $[D,Y,d,\{\mu_e\}]$ of a framed, 
$n_K$-marked graph decomposition of $H_1(\ol{S}_K,\Z/m)$, which admits $G_K$ as a group of 
automorphisms, via the natural symplectic action of $G_K$ on $H_1(\ol{S}_K,\Z/m)$, parameterizes 
a stratum of $\ccM(S_K)^{(m)}$ which is $G_K$-invariant and thus intersects the fixed point locus 
in $\ccM(S_K)^{(m)}$ for the action of $G_K$. 

Therefore, the union of the images of the maps $\mathfrak{b}_k$, for $k\geq 0$, includes all 
framed, $n_K$-marked graph decomposition endowed with a natural action of $G_K$, if and only if,
each $G_K$-invariant stratum of $\ccM(S_K)^{(m)}$ intersects {\it all} connected components of 
the fixed point locus of $G_K$ in $\ccM(S_K)^{(m)}$, which is extremely unclear.

The above geometric picture suggests at least that the theory of graph decompositions can be
used to "approximate" the nerve of Looijenga level structures in a sense to be made clear 
by the following definition and Corollary~\ref{realization}.

\begin{definition}\label{approximation}
For $2g-2+n\geq 0$, let $K$ be a finite index characteristic subgroup
of $\Pi_{g,n}$ and let $p_K\co S_K\ra S_{g,n}$ be the associated covering with covering transformation
group $G_K$. Let $n_K$ be the number of punctures on the Riemann surface $S_K$. For an integer 
$m\geq 2$, let then $\cG^{n_K}_{G_K}(H_1(\ol{S}_K,\Z/m))_\bt$ 
be the simplicial set whose semi-simplicial set of non-degenerate simplices is defined as follows.
\begin{enumerate}
\item The set of $k$-simplices is the set of equivalence classes 
$[D,Y,d,\{\mu_e\}]$ of framed, $n_K$-marked graph decompositions of $H_1(\ol{S}_K,\Z/m)$ 
endowed with a {\it natural action} of $G_K$ on the graph $Y$ such that $|e(Y)/G_K|=k+1$. 
By {\it natural action}, it is meant that, for all $v\in v(Y)$ and $f\in G_K$, it holds 
$D_{f(v)} = f(D_v)$ and $|d(f(v)|=|d(v)|$, and that, for all $e\in e(Y)$ and $f\in G_K$, it holds 
$\mu_{f(e)}(1)=\pm f(\mu_e(1))$.

\item The faces of a $k$-simplex $[D,Y,d,\{\mu_e\}]$  are the $(k-1)$-simplices $[D',Y',d',\{\mu'_e\}]$ 
such that $(D',Y',d')\leq (D,Y,d)$ and it holds $\mu'_e(1)=\pm\mu_e(1)$, for all $e\in e(Y')$.
A face map $\dd_{\mathfrak{o}}\co [D,Y,d,\{\mu_e\}]\ra [D',Y',d',\{\mu'_e\}]$ is then assigned to 
every $G_K$-orbit $\mathfrak{o}$ of unoriented edges of $Y$ such that, for some representative 
$(D',Y',d',\{\mu'_e\})$ of the equivalence class, its edge groups are the edge groups associated 
to the edges of $Y\ssm\mathfrak{o}$. 
\end{enumerate}
\end{definition}

\begin{remarks}\label{simplicial complex}
For $2g-2+n\geq 0$ and $m\geq 2$, let $K$ be a finite index characteristic 
subgroup of $\Pi_{g,n}$ which satisfies the hypotheses of Theorem~\ref{smoothness criterion II}.
Then, it holds:
\begin{enumerate}
\item For a simplex $\sg\in C(S_{g,n})_\bt$, let $(D,Y,d,\{\mu_e\})$ be a framed, $n_K$-marked 
graph decomposition of $H_1(\ol{S}_K,\Z/m)$ induced by the partition $S_K\ssm p_K^{-1}(\sg)$.
Then, the vertex and edge groups of $(D,Y,d,\{\mu_e\})$ are all distinct. Therefore,
there is only one possible action of the group $G_K$ on the graph $Y$ compatible with
its natural action on the homology group $H_1(\ol{S}_K,\Z/m)$. Moreover,
there are no non-trivial self-equivalences of $(D,Y,d,\{\mu_e\})$, so that its equivalence
class is determined modulo unique isomorphisms of the underlying graphs of its 
representatives.

\item A framed $n_K$-marked graph decomposition $(D,Y,d,\{\mu_e\})$ of $H_1(\ol{S}_K,\Z/m)$,
endowed with a natural action of $G_K$, determines a $G_K$-invariant stratum of $\ccM(S_K)^{(m)}$ 
which then parameterizes some curve endowed with a $G_K$-action. This implies that there is a 
conjugate $G_K'$ of $G_K$ inside $\GG(S_K)$ such that the $n_K$-marked graph decomposition 
$(D,Y,d,\{\mu_e\})$ is induced by a partition $S_K\ssm q^{-1}(\sg)$, where $q\co S_K\ra S_K/G_K'$ 
is the quotient map, for some $\sg\in C(S_K/G_K')_\bt$, and it holds $G_K'\equiv G_K\!\mod\GG(m)$. 
In particular, also the $n_K$-marked graph decomposition $(D,Y,d,\{\mu_e\})$ has the properties stated 
in i.).

\item It follows from i.) and ii.) that all non-degenarate simplices of 
$\cG^{n_K}_{G_K}(H_1(\ol{S}_K,\Z/m))_\bt$ are determined by their faces. In other words, 
$\cG^{n_K}_{G_K}(H_1(\ol{S}_K,\Z/m))_\bt$ is the simplicial set associated to a combinatorial 
simplicial complex.

\item For $m\geq 2$, the covering $p_K\co S_K\ra S_{g,n}$ induces a map of finite simplicial sets
$$\mathfrak{g}_{K,(m)}\co C^{K,(m)}(S_{g,n})_\bt\ra\cG^{n_K}_{G_K}(H_1(\ol{S}_K,\Z/m))_\bt$$
and the group $N_{\Sp(H_1(\ol{S}_K,\Z/m))}(G_K)$
acts naturally on $\cG^{n_K}_{G_K}(H_1(\ol{S}_K,\Z/m))_\bt$, with its 
subgroup $G_K$ acting trivially. Therefore, there is a natural action of the Teichm\"uller group 
$\GG_{g,[n]}$ on this simplicial set, via the representation $\rho_{K,(m)}$ of Section~\ref{Loo},
and the map $\mathfrak{g}_{K,(m)}$ is $\GG_{g,[n]}\left/\GG^{K,(m)}\right.$-equivariant.
\end{enumerate}
\end{remarks}

By Theorem~\ref{comparison}, the Looijenga level $\GG^{K,(m)}$ is contained in the geometric
level $\GG^K$. Therefore, there is also a natural surjective map 
$\pi_K\co C^{K,(m)}(S_{g,n})_\bt\tura C^K(S_{g,n})_\bt$. The maps $\mathfrak{g}_{K,(m)}$ and
$\pi_K$ are related in the following way:

\begin{theorem}\label{factorization}For $2g-2+n\geq 0$, let $K$ be a finite index characteristic 
subgroup of $\Pi_{g,n}$ which satisfies the hypotheses of Theorem~\ref{smoothness criterion II}. 
Then, for $m\geq 2$, the natural surjective map 
$\pi_K\co C^{K,(m)}(S_{g,n})_\bt\tura C^K(S_{g,n})_\bt$ factors through the natural surjective map
$\mathfrak{g}_{K,(m)}\co C^{K,(m)}(S_{g,n})_\bt\tura\Im\,\mathfrak{g}_{K,(m)}$ and a surjective
map $\Im\,\mathfrak{g}_{K,(m)}\tura C^K(S_{g,n})_\bt$.
\end{theorem}

\begin{proof}The map $\mathfrak{g}_{K,(m)}$ is induced by the map 
$C(S_{g,n})_\bt\ra\cG^{n_K}_{G_K}(H_1(\ol{S}_K,\Z/m))_\bt$ which associates to a simplex
$\sg\in C(S_{g,n})_\bt$ the equivalence class of the framed, $n_K$-marked graph decomposition 
of $H_1(\ol{S}_K,\Z/m)$ induced by the decomposition $S_K\ssm p_K^{-1}(\sg)$ of $S_K$, where 
$p_K\co S_K\ra S_{g,n}$ is the Galois covering associated to the subgroup $K$ of $\Pi_{g,n}$.

By Remarks~\ref{simplicial complex}, the hypotheses of Theorem~\ref{smoothness criterion II}
on $K$ imply that, if two framed, $n_K$-marked graph decompositions of $H_1(\ol{S}_K,\Z/m)$,
endowed with a natural $G_K$-action, are equivalent, there is a unique isomorphism of the 
underlying graphs which induces the equivalence. 
Therefore, we can identify them by means of this isomorphism.

Let then be given two simplices $\sg,\sg'\in C(S_{g,n})_\bt$ inducing the same framed, 
$n_K$-marked graph decomposition $(D,Y,d,\{\mu_e\})$ of $H_1(\ol{S}_K,\Z/m)$.
In order to prove the theorem, we have to show that there is an $f\in Z_{\GG(S_K)}(G_K)$ 
such that $f(p_K^{-1}(\sg))=p_K^{-1}(\sg')$. By the short exact sequence $(2)$ in 
Section~\ref{Loo},  the image $\ol{f}$ of $f$, by the natural epimorphism 
$N_{\GG(S_K)}(G_K)\tura\GG_{g,[n]}$, is then contained in the level $\GG^K$ 
and is such that $\ol{f}(\sg)=\sg'$.

The $n_K$-marked graph $Y$ determines the topological types of both $S_K\ssm p_K^{-1}(\sg)$ 
and $S_K\ssm p_K^{-1}(\sg')$ and then of $S_{g,n}\ssm\sg$ and $S_{g,n}\ssm\sg'$ as well.
So, at least, we know,  that there is an element $f\in N_{\GG(S_K)}(G_K)$ such that 
$f(p_K^{-1}(\sg))=p_K^{-1}(\sg')$.

Let $p_K^{-1}(\sg)=\cup_{i=0}^k\gm_i$ and $p_K^{-1}(\sg')=\cup_{i=0}^k\gm'_i$. Let us
assume that $\gm_i$ and $\gm_i'$ define the same cycle in $H_1(\ol{S}_K,\Z/m)$, 
for $i=0,\ldots,k$. The hypothesis on the subgroup $K$ implies that $\gm_i$ and $\gm_j$ 
define distinct homology classes whenever $i\neq j$ and then that the action of the group $G_K$ 
on $H_1(\ol{S}_K,\Z/m)$ determines its action on the closed submanifolds $p_K^{-1}(\sg)$ and 
$p_K^{-1}(\sg')$. So, there is a unique $G_K$-equivariant homeomorphism 
$\delta\co p_K^{-1}(\sg)\ra p_K^{-1}(\sg')$ which sends $\gm_i$ to $\gm_i'$, 
for $i=0,\ldots,k$. We claim that $\delta$ can be extended 
to a $G_K$-equivariant homeomorphism $\chi\co S_K\ra S_K$. 

By Theorem~\ref{abelian nerve m}, the simplices $\sg$ and $\sg'$ determine the same
stratum of the abelian level structure $\ccM(S_K)^{(m)}$. Therefore,
there is an element $\phi$ in the abelian level $\GG(m)$ of $\GG(S_K)$ such that it holds 
$\phi(p_K^{-1}(\sg))=p_K^{-1}(\sg')$ and then, by the above assumptions, such that 
$\phi(\gm_i)=\gm_i'$, for $i=0,\ldots,k$. 

Let $G_K':=\phi G_K\phi^{-1}<\GG(S_K)$
and let $\inn\,\phi\co G_K\ra G_K'$ be the isomorphism defined by the assignment 
$\alpha\mapsto\phi\alpha\phi^{-1}$. The homeomorphism $\phi$ then satisfies 
$\phi(\alpha\cdot x)=\inn\,\phi(\alpha)\cdot\phi(x)$, for all $x\in S_K$ and $\alpha\in G_K$.
The group $G_K'$, like the group $G_K$, preserves the union of circles $p_K^{-1}(\sg')$. 
More precisely, it holds $\alpha\cdot x=\inn\,\phi(\alpha)\cdot x$, for all $x\in p_K^{-1}(\sg')$ 
and $\alpha\in G_K$.

Let $S_K':=S_K\ssm p_K^{-1}(\sg')$. By the definition of $\phi$ and $G_K'$, 
the natural maps  $S_K'\ra S_K'/G_K$ and 
$S_K'\ra S_K'/G_K'$ are equivalent Galois \'etale coverings of homeomorphic surfaces.
Thus, the actions of $G_K$ and $G_K'$ on $S_K'$ are conjugated by a self-homeomorphism 
of $S_K'$, which lifts a homeomorphism $S_K'/G_K\ra S_K'/G_K'$.

The $G_K$-covering $S_K\ra S_K/G_K$ and the $G_K'$-covering $S_K\ra S_K/G_K'$  
are determined by their restrictions on the subspace $p_K^{-1}(\sg')$, 
modulo the actions by conjugation, respectively, of the groups $N_{\GG(S_K)}(G_K)$ and 
$N_{\GG(S_K)}(G_K')$ on $S_K$. Therefore, the identity map on the subspace $p_K^{-1}(\sg')$ 
extends to a homeomorphism $\psi\co S_K\ra S_K$ which lifts a homeomorphism 
$S_K/G_K\ra S_K/G_K'$ and satisfies $\psi(\alpha\cdot x)=\inn\,\phi(\alpha)\cdot\psi(x)$, 
for all $x\in S_K$ and $\alpha\in G_K$. 
 
The composition $\chi:=\psi^{-1}\circ\phi\co S_K\ra S_K$ is then a $G_K$-equivariant 
homeomorphism with the required property that $\chi(p_K^{-1}(\sg))=p_K^{-1}(\sg')$.

\end{proof}

\begin{corollary}\label{injectivety}For $2g-2+n\geq 0$, let $\{K\}$ be a cofinal system of finite index 
characteristic subgroups of $\Pi_{g,n}$ satisfying the hypotheses of 
Theorem~\ref{smoothness criterion II} and let $\{p_K\co S_K\ra S_{g,n}\}$ be the associated set 
of coverings. For any fixed integer $m\geq 2$, there is then a 
continuous injective $\cGG_{g,n}$-equivariant map of simplicial profinite sets:
$$\prod_{K}\check{\mathfrak{g}}_{K,(m)}\co\check{C}(S_{g,n})_\bt\hookra
\prod_{K}\cG^{n_K}_{G_K}(H_1(\ol{S}_K,\Z/m))_\bt.$$
\end{corollary}

\begin{proof}The map $\check{\mathfrak{g}}_{K,(m)}\co\check{C}(S_{g,n})_\bt\ra
\cG^{n_K}_{G_K}(H_1(\ol{S}_K,\Z/m))_\bt$ is defined by the composition of the natural projection
$\check{C}(S_{g,n})_\bt\tura C^{K,(m)}(S_{g,n})_\bt$ with the map $\mathfrak{g}_{K,(m)}$
associated to the covering $p_K\co S_K\ra S_{g,n}$. The corollary then
follows from Corollary~\ref{cofinal} and Theorem~\ref{factorization}.

\end{proof}

\begin{remark}\label{cofinal2}By Lemma~\ref{no cut-pairs}, a cofinal system of finite index 
characteristic subgroups of $\Pi_{g,n}$ satisfying the hypotheses of 
Theorem~\ref{smoothness criterion II} can be constructed from any given cofinal system of finite
index characteristic subgroups of $\Pi_{g,n}$. 
\end{remark}

It is not clear whether the set $\{\cG^{n_K}_{G_K}(H_1(\ol{S}_K,\Z/m))_\bt\}$, where 
$\{K\}$ is a cofinal system of finite index characteristic subgroups of $\Pi_{g,n}$ and $m\geq 2$ a
fixed integer, forms an inverse system of finite simplicial sets. This is partially fixed by
the following formal consequence of Corollary~\ref{injectivety}:

\begin{corollary}\label{realization}For $2g-2+n\geq 0$, let $\{K\}$ 
be a cofinal system of finite index characteristic subgroups of $\Pi_{g,n}$ and let 
$\{p_K\co S_K\ra S_{g,n}\}$ be the associated set of coverings. 
Then, for any fixed integer $m\geq 2$, there is 
a cofinal sub-system $\{K_\ld\}_{\ld\in\Ld}$ of $\{K\}$ such that the set 
$\{\Im\,\mathfrak{g}_{K_\ld,(m)}\}_{\ld\in\Ld}$ can be organized in an inverse system of finite 
simplicial sets and there is a $\cGG_{g,n}$-equivariant isomorphism of simplicial profinite sets:
$$\check{C}(S_{g,n})_\bt\cong\lim\limits_{\sr{\textstyle\st\longleftarrow}{\sst \ld\in\Ld}}
\Im\,\mathfrak{g}_{K_\ld,(m)}.$$
\end{corollary}

\begin{remark}\label{simplicial complex2}
For $K$ as in the hypotheses of Theorem~\ref{smoothness criterion II}, 
by i.) Remarks~\ref{simplicial complex}, the simplicial set $\Im\,\mathfrak{g}_{K,(m)}$ can be 
realized as the simplicial set associated to an ordered simplicial complex.
The same property then holds for the simplicial profinite set $\check{C}(S_{g,n})_\bt$, that is to say,
there is an ordered simplicial complex $\check{C}(S_{g,n})$ whose sets of $k$-simplices 
are profinite, for all $k\geq 0$, and whose associated simplicial set is $\check{C}(S_{g,n})_\bt$.
\end{remark}

By Corollary~\ref{realization} and Corollary~\ref{cong}, the congruence subgroup property 
would follow from a positive answer to the following question in combinatorial topology: 

\begin{question}\label{question}For $g\geq 2$ and some integer $m\geq 2$, is there a cofinal 
system $\{K\}$ of finite index characteristic subgroups of $\Pi_{g}$ such that the finite simplicial 
sets $\Im\,\mathfrak{g}_{K,(m)}$ are simply connected?
\end{question}

Of course, it would be much easier to approach Question~\ref{question} if a more explicit
description of the images of the maps $\mathfrak{g}_{K,(m)}$ were available, for instance,
if the maps $\mathfrak{g}_{K,(m)}$ were known to be surjective for some cofinal system $\{K\}$
and some integer $m\geq 2$.

If the answer to Question~\ref{question} is positive, by Corollary~\ref{realization}, 
there is a cofinal system $\{K_\ld\}_{\ld\in\Ld}$ of finite index characteristic 
subgroups of $\Pi_{g}$, refining the cofinal system $\{K\}$, such that 
$\{\Im\,\mathfrak{g}_{K_\ld,(m)}\}$ forms an inverse system of simply connected finite simplicial 
sets and, then, it holds:
$$\hat{\pi}_1(\check{C}(S_{g})_\bt)\cong\lim\limits_{\sr{\textstyle\st\longleftarrow}
{\sst \ld\in\Ld}}\hat{\pi}_1(\Im\,\mathfrak{g}_{K_\ld,(m)})=\{1\}.$$

For the moment, the only scarce evidence for a positive answer to Question~\ref{question} is 
provided by Proposition~\ref{BP}. Compared to the approach to the subgroup congruence problem 
of \cite{PFT}, this has at least the advantage that it can be dealt with 
standard techniques of combinatorial topology in the spirit of \cite{Kallen} and \cite{Putman1}.

\subsection*{Acknowledgements}
This work was begun during my stay at the Department of Mathematics of the University of Costa 
Rica in San Jos\'e and completed during my stay at the Department of Mathematics of the
University of los Andes in Bogot\'a. I thank both institutions for their support, financial and otherwise. 
I also thank A. Putman for suggesting the second part of Proposition~\ref{BP}.

\bigskip

\noindent Address:\, Departamento de Matem\'aticas, Universidad de los Andes, \\
\hspace*{1.7cm}  Carrera $1^a$ $\mathrm{N}^o$ 18A-10, Bogot\'a, Colombia.
\\
E--mail:\,\,\, marco.boggi@gmail.com

\end{document}